\documentclass[reqno,11pt]{amsart}

\usepackage[left=3cm, right=3cm, top=3cm, bottom=3cm]{geometry}
\usepackage{amssymb}
\usepackage{amsmath}
\usepackage{color}

\usepackage{cite}

\theoremstyle{plain}
\newtheorem{theorem}{Theorem}[section]
\newtheorem*{theorem*}{Theorem}
\newtheorem{proposition}{Proposition}[section]
\newtheorem{lemma}{Lemma}[section]
\newtheorem{corollary}{Corollary}[section]
\newtheorem{remark}{Remark}[section]

\theoremstyle{definition}
\newtheorem{definition}{Definition}[section]

\theoremstyle{remark}
\newtheorem{example}{Example}[section]

\begin{document}

\title{Eigenfunction expansions associated with operator differential equation depending on spectral parameter nonlinearly}

\author{Volodymyr Khrabustovskyi}

\address{Ukrainian State Academy of Railway Transport,  Kharkiv, Ukraine}

\email{v{\_}khrabustovskyi@ukr.net}


\keywords{Relation generated by pair of differential expressions
one of which depends on spectral parameter in nonlinear manner,
non-injective resolvent, generalized resolvent}

\begin{center}
{\Large \textbf{Eigenfunction expansions associated with operator differential equation depending on spectral parameter nonlinearly}}\\\medskip {\large Volodymyr Khrabustovskyi}\\\medskip
{\small Ukrainian
State Academy of Railway Transport,  Kharkiv, Ukraine}\\
{\small v{\_}khrabustovskyi@ukr.net}
\end{center}\vspace{13px}

\begin{flushright}
\noindent \textit{Dedicated to my teacher Professor F.S. Rofe-Beketov\\ on the occasion of his jubilee}\bigskip
\end{flushright}

{\small \noindent \textbf{Abstract.} For operator differential equation which depends
on the spectral parameter in the Nevanlinna manner we obtain the expansions in eigenfunctions.\bigskip

\noindent \textbf{Keywords and phrases:} Relation generated by
pair of differential expressions one of which depends on spectral
parameter in nonlinear manner, non-injective resolvent,
generalized resolvent, Weyl type operator function and solution, eigenfunction expansion.\medskip

}\bigskip

\section*{Introduction}
We consider either on finite or infinite interval operator
differential equation of arbitrary order
\begin{gather}
\label{GEQ__1_} l_\lambda[y]=m[f],\ t\in\bar{\mathcal{I}},\
\mathcal{I}=(a,b)\subseteq\mathbb{R}^1
\end{gather}
in the space of vector-functions with values in the separable
Hilbert space $\mathcal{H}$, where
\begin{gather}
\label{GEQ__2_} l_\lambda[y]=l[y]-\lambda m[y]-n_\lambda[y],
\end{gather}
$l[y],m[y]$ are symmetric operator differential expression. The
order of $l_\lambda[y]$ is equal to $r>0$. For the expression
$m[y]$ the subintegral quadratic form $m\{y,y\}$ of its Dirichlet
integral $m[y,y]=\int_{\mathcal{I}}m\{y,y\}dt$ is nonnegative for
$t\in\bar{\mathcal{I}}$. The leading coefficient of the expression
$m[y]$ may lack the inverse from $B(\mathcal{H})$ for any
$t\in\bar{\mathcal{I}}$ and even it may vanish on some intervals.
For the operator differential expression $n_\lambda[y]$ the form
$n_\lambda\{y,y\}$ depends on $\lambda$ in the Nevanlinna manner
for $t\in\bar{\mathcal{I}}$. Therefore the order $s\geq 0$ of
$m[y]$ is even and $\leq r$.

In paper \cite{KhrabMAG} (see also \cite{KhrabArxiv}) in the Hilbert space $L^2_m(\mathcal{I})$ with metrics generated
by the form $m[y,y]$ for equation
(\ref{GEQ__1_})-(\ref{GEQ__2_}) we constructed analogs
$R(\lambda)$ of the generalized resolvents which in general are
non-injective and which possess the following representation:
\begin{gather}\label{GEQ__3_}
R(\lambda)=\int_{\mathbb{R}^1}{dE_\mu\over \mu-\lambda}
\end{gather}
where $E_\mu$ is a generalized spectral family for which
$E_\infty$ is less or equal to the identity operator.

This analogue in an integro-differential operator depending on the characteristic operator of
the equation
\begin{gather}\label{GEQ__5_}
l_\lambda[y]=-{(\Im l_\lambda)[f]\over \Im\lambda},\
t\in\bar{\mathcal{I}},
\end{gather}
where $(\Im l_\lambda)[f]={1\over 2i}(l[f]-l^*[f])$. This characteristic operator was defined in \cite{KhrabArxiv,KhrabMAG}. It is an analogue of the characteristic matrix of scalar differential operator \cite{Shtraus1} (see also \cite[p. 280]{Naimark}). The operator $R(\lambda)$ in the case $n_\lambda[y]\equiv 0$ is the generalized resolvent of the minimal relation corresponding to equation \eqref{GEQ__1_} (see the details in \cite{Khrab6,KhrabMAG,KhrabArxiv}).

In this paper we calculate $E_{\Delta } $ and derive an inequality of Bessel type. In the case when the expression $n_{\lambda } \left[y\right]$ submits in a special way to the expression $m\left[y\right]$ we obtain the inversion formulae and the Parseval equality. The general results obtained in the work are illustrated on the example of equation \eqref{GEQ__1_} with coefficients which are periodic on the semi-axes. We remark that in the case $n_{\lambda } \left[y\right]\equiv 0$ it follows from  \cite{Khrab3,Khrab6} that if $\mathcal{I}={\mathbb R}^{1} ,\, r>s$ and $\dim \mathcal{H}<\infty $ then $E_{\mu } $ for equation \eqref{GEQ__1_} with periodic coefficients has no jumps. (For $r=s$ in the described case $E_{\mu } $ may have jump (see e.g. \cite{Khrab6})). We show that in contrast to the case $n_{\lambda } \left[y\right]\equiv 0$ if $r=1,\, \dim \mathcal{H}=2$ then $E_{\mu } $ for equation \eqref{GEQ__1_} with periodic coefficients on the axis may have jump.

In the case $n_{\lambda } \left[y\right]\equiv 0$ the eigenfunction expansion results above are obtained in paper \cite{Khrab6}, which contains its comparison with results, which was obtained earlier for this case.

Eigenfunction expansions for differential operators and relations are considered in the monographs \cite{CodLev,DS,Atkinson,Berez1,Berez2,LyaSto,Marchenko,RBKholkin,Sakhno}. Let us notice that for infinite systems first eigenfunction expansion results are obtained in \cite{RB60} for operator Sturm-Liouvill equation. (Later it was done in \cite{Gor} in another way). The expansion in eigenfunctions of operator equation of highest order (analogous to scalar case \cite{Shtraus1}) was obtained in \cite{Bruk1}).

Also for the case of the half-axis we obtain for equation \eqref{GEQ__1_} a generalization of the result from \cite{Shtraus2} on the expansion in solutions of scalar Sturm-Liouville equation which satisfy in the regular end-point the boundary condition depending on a spectral parameter. To do this we introduce for equation \eqref{GEQ__1_} Weyl type functons and solutions. Such solutions for operator equation of first order containing spectral parameter in Nevanlinna manner was constructed in \cite{Khrab5}. For finite canonical systems
the parametrisation of Weyl functions for such equation was obtained in \cite{Orlov} with the help of $\mathcal{J}$-theory. For these systems depending on spectral parameter in linear manner, such parametrization in another form was obtained in \cite{0} with the help of abstract Weyl function. The method of studying of differential operators and relations based on use of the abstract Weyl function and its generalization (Weyl family) was proposed in \cite{DerMalamud,DHMdS1,DHMdS2}.

A part of the results of this paper is contained in the preliminary form in preprint \cite{KhrabArxiv}. 

We denote by $(\ .\ )$ and $\|\cdot\|$ the scalar product and the
norm in various spaces with special indices if it is necessary. For differential operation $l$ we denote 
$\Re l={1\over 2}(l+l^*)$, $\Im l={1\over 2i}(l-l^*)$.

Let an interval $\Delta \subseteq \mathbb{R}^1 ,\,
f\left(t\right)\, \left(t\in \Delta \right)$ be a function with
values in some Banach space $B$. The notation $f\left(t\right)\in
C^{k} \left(\Delta,B\right),\ k=0,\, 1,\, ...$ (we omit the index
$k$ if $k=0$) means, that in any point of $\Delta$
$f\left(t\right)$ has continuous in the norm  $\left\| \, \cdot \,
\right\| _{B}$ derivatives of order up to and including $l$ that
are taken in the norm $\left\| \, \cdot \,\right\| _{B} $; if
$\Delta$ is either semi-open or closed interval then on its ends
belonging to $\Delta$ the one-side continuous derivatives exist.
The notation $f\left(t\right)\in C_{0}^{k} \left(\Delta,B\right)$
means that $f\left(t\right)\in C^{k} \left(\Delta,B \right)$ and
$f\left(t\right)=0$ in the neighbourhoods of the ends of $\Delta$.


\section{Characteristic operator. Weyl type operator function and solution}

In order to formulate the eigenfunction expansion results we present in this section several results from \cite{KhrabMAG} (see also \cite{KhrabArxiv}). Comparing with \cite{KhrabMAG,KhrabArxiv} some of these results are given here  in either more general or more weak form. Lemmas \ref{lm11},  \ref{lm12}, Theorem \ref{th12} and Corollary \ref{cor11} are new.

We consider an operator differential equation in separable Hilbert
space $\mathcal{H}_{1} $:
\begin{equation} \label{GEQ__46_}
\frac{i}{2} \left(\left(Q\left(t\right)x\left(t\right)\right)^{{'}
} +Q^{*} \left(t\right)x'\left(t\right)\right)-H_{\lambda }
\left(t\right)x\left(t\right)=W_{\lambda }
\left(t\right)F\left(t\right),\quad t\in \bar{\mathcal{I}},
\end{equation}
where $Q\left(t\right),\, \left[\Re \, Q\left(t\right)\right]^{-1}
,\, H_{\lambda } \left(t\right)\in B\left(\mathcal{H}_{1}
\right),\, Q\left(t\right)\in C^{1}
\left(\bar{\mathcal{I}},B\left(\mathcal{H}_{1} \right)\right)$;
the operator function $H_{\lambda } \left(t\right)$ is continuous
in $t$ and is Nevanlinna's in $\lambda $. Namely the following
condition holds:

(\textbf{A}) The set $\mathcal{A}\supseteq \mathbb{C}\setminus
\mathbb{R}^1$ exists, any its point have a neighbourhood
independent of $t\in \bar{\mathcal{I}}$, in this neighbourhood
$H_{\lambda } \left(t\right)$ is analytic $\forall t\in
\bar{\mathcal{I}};\, \forall \lambda \in \mathcal{A}\, H_{\lambda
} \left(t\right)=H^*_{\bar \lambda}(t)\in
C\left(\bar{\mathcal{I}},B\left(\mathcal{H}_1\right)\right)$; the
weight $W_{\lambda } \left(t\right)=\Im H_{\lambda }
\left(t\right)/\Im \lambda \ge 0\left(\Im  \lambda \ne 0\right)$.

In view of \cite{Khrab5} $\forall \mu \in \mathcal{A}\bigcap
\mathbb{R}^1:\, W_{\mu } \left(t\right)=\left.\partial H_{\lambda}
\left(t\right)/\partial \lambda\right|_{\lambda=\mu} $ is Bochner
locally integrable in the uniform operator topology.

For convenience we suppose that $0\in \bar{\mathcal{I}}$ and we
denote $\Re \, Q\left(0\right)=G$.

Let $X_{\lambda } \left(t\right)$ be the operator solution of
homogeneous equation \eqref{GEQ__46_} satisfying the initial
condition $X_{\lambda } \left(0\right)=I$, where $I$ is an
identity operator in $\mathcal{H}_{1} $. 

For any $\alpha ,\, \beta \in \bar{\mathcal{I}},\, \alpha \le
\beta $ we denote $\Delta _{\lambda } \left(\alpha ,\beta
\right)=\int _{\alpha }^{\beta }X_{\lambda }^{*}
\left(t\right)W_{\lambda } \left(t\right)X_{\lambda }
\left(t\right) dt$,\\ $N=\left\{h\in \mathcal{H}_{1} \left|h\in
Ker\Delta _{\lambda } \left(\alpha ,\beta \right)\right. \forall
\alpha ,\beta \right\},P$ is the ortho-projection onto $N^{\bot }
$. $N$ is independent of $\lambda \in \mathcal{A}$ \cite{Khrab5}.

For $x\left(t\right)\in \mathcal{H}_{1} $  we denote
$U\left[x\left(t\right)\right]=\left(\left[\Re \,
Q\left(t\right)\right]x\left(t\right),x\left(t\right)\right)$.

\begin{definition}\cite{Khrab4,Khrab5}
An analytic operator-function $M\left(\lambda \right)=M^{*}
\left(\bar{\lambda }\right)\in B\left(\mathcal{H}_{1} \right)$ of
non-real $\lambda $ is called a characteristic operator of
equation \eqref{GEQ__46_} on $\mathcal{I}$,
if for $\Im  \lambda \ne 0$ and for any $\mathcal{H}_{1} $ -
valued vector-function $F\left(t\right)\in L_{W_{\lambda } }^{2}
\left(\mathcal{I}\right)$ with compact support the corresponding
solution $x_{\lambda } \left(t\right)$ of equation
\eqref{GEQ__46_} of the form
\begin{equation} \label{GEQ__47_}
x_{\lambda } \left(t,F\right)=\mathcal{R}_\lambda F=\int
_{\mathcal{I}}X_{\lambda } \left(t\right) \left\{M\left(\lambda
\right)-\frac{1}{2} sgn\left(s-t\right)\left(iG\right)^{-1}
\right\}X_{\bar{\lambda }}^{*} \left(s\right)W_{\lambda }
\left(s\right)F\left(s\right)ds
\end{equation}
satisfies the condition
\begin{gather}\label{GEQ__47++_}
\left(\Im  \lambda \right)\mathop{\lim }\limits_{\left(\alpha
,\beta \right)\uparrow \mathcal{I}} \left(U\left[x_{\lambda }
\left(\beta ,F\right)\right]-U\left[x_{\lambda } \left(\alpha
,F\right)\right]\right)\le 0,\ \Im  \lambda \ne
0.
\end{gather}
\end{definition}

Let us note that in \cite{Khrab5} characteristic operator  was defined if
$Q(t)=Q^*(t)$. Our case is equivalent to this one since equation
\eqref{GEQ__46_} coincides with equation of
\eqref{GEQ__46_} type with $\Re Q(t)$ instead of $Q(t)$ and
with $H_\lambda(t)-{1\over 2}\Im Q'(t)$ instead of $H_\lambda(t)$.

The properties of characteristic operator and sufficient conditions of the characteristic operators
existence are obtained in \cite{Khrab4,Khrab5}.

In the case $\mathrm{dim}\mathcal{H}_1<\infty$,
$Q(t)=\mathcal{J}=\mathcal{J}^*=\mathcal{J}^{-1}$, $-\infty<a=0$
the description of characteristic operators was obtained in \cite{Orlov} (the
results of \cite{Orlov} were specified and supplemented in
\cite{Khrab+}). In the case $\mathrm{dim} \mathcal{H}_1=\infty$
and $\mathcal{I}$ is finite the description of characteristic operators was obtained
in \cite{Khrab5}. These descriptions are obtained under the
condition that
\begin{gather}\label{star8}
\exists\lambda_0\in \mathcal{A},\ [\alpha,\beta]\subseteq
\overline{\mathcal{I}}:\ \Delta_{\lambda_0}(\alpha,\beta)\gg 0.
\end{gather}

\begin{definition}\label{def+}
\cite{Khrab4,Khrab5} Let  $M\left(\lambda \right)$ be the characteristic operator of
equation (\ref{GEQ__46_}) on  $\mathcal{I} $. We say that the
corresponding condition (\ref{GEQ__47++_}) is separated for
nonreal $\lambda =\mu _{0} $ if for any $\mathcal{H}_1$-valued
vector function $f\left(t\right)\in L_{W_{\mu_{0} }(t) }^{2}
\left(\mathcal{I} \right)$ with compact support the following
inequalities holds simultaneously for the solution $x_{\mu _{0} }
\left(t\right)$ (\ref{GEQ__47_}) of equation
(\ref{GEQ__46_}):
\begin{gather}
\label{12} \displaystyle \lim\limits_{\alpha\downarrow a}\Im\mu
_{0} U\left[x_{\mu _{0} } \left(\alpha\right)\right]\ge 0,\quad
\mathop{\lim }\limits_{\beta \uparrow b} \Im\mu _{0} U\left[x_{\mu
_{0} } \left(\beta \right)\right]\le 0.
\end{gather}
\end{definition}

\begin{theorem}\cite{Khrab4,Khrab5}\label{th++}
Let $P=I$, $M\left(\lambda \right)$ be the characteristic operator of equation
(\ref{GEQ__46_}), $\mathcal{P}(\lambda)=iM(\lambda)G+{1\over 2}I$, so that we have the following representation
\begin{gather}
\label{13} \displaystyle M\left(\lambda \right)=\left(\mathcal{P}
\left(\lambda \right)-\frac{1}{2} I\right)\left(iG\right)^{-1}.
\end{gather}

Then the condition (\ref{GEQ__47++_}) corresponding to
$M\left(\lambda \right)$ is separated for $\lambda =\mu _{0} $ if
and only if the operator $\mathcal{P} \left(\mu_{0} \right)$ is
the projection, i.e.
\begin{gather*}
 \displaystyle \mathcal{P} \left(\mu _{0}
\right)=\mathcal{P} ^{2} \left(\mu _{0} \right).
\end{gather*}
\end{theorem}

\begin{definition}\cite{Khrab4,Khrab5}\label{def++}
If the operator-function $M\left(\lambda \right)$ of the form
(\ref{13}) is the characteristic operator of equation (\ref{GEQ__46_}) on
$\mathcal{I} $ and, moreover, $\mathcal{P} \left(\lambda
\right)=\mathcal{P}^{2} \left(\lambda \right)$, then
$\mathcal{P}\left(\lambda \right)$ is called a characteristic
projection of equation (\ref{GEQ__46_}) on $\mathcal{I}$.
\end{definition}

The properties of characteristic projections and sufficient conditions for their
existence are obtained in \cite{Khrab5}. Also \cite{Khrab5}
contains the description of characteristic projections and abstract analogue of
Theorem \ref{th++}. Necessary and sufficient conditions for existence of characteristic operator, which corresponds to such separated boundary conditions that corresponding boundary condition in
regular point is self-adjoint, are obtained in \cite{KhrabMAG} (see also \cite{KhrabArxiv})  with the help of Theorem \ref{th++}. In the case of self-adjoint boundary conditions the analogue of
this result for regular differential operators in space of
vector-functions was proved in \cite{RB} (see also
\cite{RBKholkin}). For finite canonical systems depending on
spectral parameter in a linear manner such analogue was proved in
\cite{Mogil}. These analogues were obtained in a different way
comparing with the proof in \cite{KhrabMAG,KhrabArxiv}.


{

From this point and till the end of Corollary \ref{cor11} we suppose that $\mathcal{H}_{1} =\mathcal{H}^{2n}$,   
\begin{gather}\label{Q}
Q(t)=\left(\begin{matrix}0&iI_n\\-iI_n& 0\end{matrix}\right)=J/i,
\end{gather}
where $I_n$ is the identity operator in $\mathcal{H}^{n}$, 
$\mathcal{I}=(0,b)$, $b\leq\infty$ and condition \eqref{star8} holds. Let condition \eqref{GEQ__47++_} be separated and $\mathcal{P}(\lambda)$ be a corresponding characteristic projection. In view of \cite[p. 469]{Khrab5} the Nevanlinna pair $\left\{-a\left(\lambda \right),\, b\left(\lambda \right)\right\},\, a(\lambda),b(\lambda)\in B\left(\mathcal{H}^{n} \right)$ (see for example \cite{DHMdS2}) and Weyl function $m\left(\lambda \right)\in B\left(\mathcal{H}^{n} \right)$ of equation \eqref{GEQ__46_} on $\mathcal{I}$ \cite{Khrab5} exist such that
\begin{gather} \label{GEQ__64ad1_} 
\mathcal{P}\left(\lambda \right)=\left(\begin{array}{c} {I_{n} } \\ {m\left(\lambda \right)} \end{array}\right)\left(b^{*} \left(\bar{\lambda }\right)-a^{*} \left(\bar{\lambda }\right)m\left(\lambda \right)\right)^{-1} \left(a_{2}^{*} \left(\bar{\lambda }\right),\, -a_{1}^{*} \left(\bar{\lambda }\right)\right),  
\\ \label{GEQ__65ad1_} 
I-\mathcal{P}\left(\lambda \right)=\left(\begin{array}{c} {a\left(\lambda \right)} \\ {b\left(\lambda \right)} \end{array}\right)\left(b\left(\lambda \right)-m\left(\lambda \right)a\left(\lambda \right)\right)^{-1} \left(-m\left(\lambda \right),I_{n} \right),\\\label{13+} 
\left(b^{*} \left(\bar{\lambda }\right)-a^{*} \left(\bar{\lambda }\right)m\left(\lambda \right)\right)^{-1} ,\, \, \left(b\left(\lambda \right)-m\left(\lambda \right)a\left(\lambda \right)\right)^{-1} \in B\left(\mathcal{H}^{n} \right).
\end{gather}
(Conversely \cite{Khrab5} $\mathcal{P}\left(\lambda \right)$ \eqref{GEQ__64ad1_} is a characteristic projection for any Nevanlinna pair $\left(-a\left(\lambda \right),\, b\left(\lambda \right)\right)$ and any Weyl function $m\left(\lambda \right)$ of equation \eqref{GEQ__46_} on $\mathcal{I}$.)

Let also domain ${D}\subseteq {\mathbb C}_{+} $ be such that $\forall \lambda \in {D}:\, \, 0\in \rho \left(a\left(\lambda \right)-ib(\lambda)\right)$ (for example ${D}={\mathbb C}_{+} $ if $\exists\lambda_\pm\in\mathbb{C}_\pm$ such that $a^{*} \left(\lambda_\pm \right)b\left(\lambda_\pm \right)=b^{*} \left(\lambda_\pm\right)a\left(\lambda_\pm \right)$ ). Let domain ${ D}_{1} $ be symmetric to ${D}$ with respect to real axis. Then the operator $\mathcal{R}_{\lambda } F$ \eqref{GEQ__47_} for $\lambda \, \in { D}\bigcup { D}_{1} $ can be represented in the following form with using the operator solulion $U_{\lambda } \left(t\right)\in B\left(\mathcal{H}^{n} ,\, \mathcal{H}^{2n} \right)$ of equation \eqref{GEQ__46_}, ($F=0$) satisfying accumulative (or dissipative) initial condition and operator solution $V_{\lambda } \left(t\right)\in B\left(\mathcal{H}^{n} ,\, \mathcal{H}^{2n} \right)$ of Weyl type of the same equation. More precisely the following proposition holds.

\begin{proposition}\label{rm21}Let $\lambda \in {D}\bigcup {D}_{1} $, $\mathcal{H}_1$-valued $F(t)\in L^2_{W_\lambda}(\mathcal{I})$\footnote{Norms $\|\cdot\|_{L^2_{W_\lambda}(\mathcal{I})}$ are equivalent for $\lambda\in \mathcal{A}$ \cite{Khrab5}.}. Then solution \eqref{GEQ__47_} of equation \eqref{GEQ__46_} is equal to
\begin{gather}\label{rlambdaf}
\mathcal{R}_{\lambda } F=\int _{0}^{t}V_{\lambda } \left(t\right)U_{\bar{\lambda }}^{*} \left(s\right)W_{\lambda } \left(s\right)F\left(s\right)ds +\int _{t}^{b}U_{\lambda } \left(t\right)V_{\bar{\lambda }}^{*} \left(s\right)W_{\lambda } \left(s\right)F\left(s\right)ds,
\end{gather}
where the integrals converge strongly if the interval of integration is infinite.
Here
\begin{equation} \label{GEQ__66ad1_} 
U_{\lambda } \left(t\right)=X_{\lambda } \left(t\right)\left(\begin{array}{c} {a\left(\lambda \right)} \\ {b\left(\lambda \right)} \end{array}\right),\, \, V_{\lambda } \left(t\right)=X_{\lambda } \left(t\right)\left(\begin{array}{c} {b\left(\lambda \right)} \\ {-a\left(\lambda \right)} \end{array}\right){K} ^{-1} \left(\lambda \right)+U_{\lambda } \left(t\right)m_{a,b} \left(\lambda \right),  
\end{equation} 
where
\begin{gather} \label{GEQ__67ad1_} 
{K} \left(\lambda \right)=a^{*} \left(\bar{\lambda }\right)a\left(\lambda \right)+b^{*} \left(\bar{\lambda }\right)b\left(\lambda \right),\, \, {K} ^{-1} \left(\lambda \right)\in B\left(\mathcal{H}^{n} \right),  
\\\label{GEQ__68ad1_} 
m_{a,b} \left(\lambda \right)=m_{a,b}^{*} \left(\bar{\lambda }\right)={K} ^{-1} \left(\lambda \right)\left(a^{*} \left(\bar{\lambda }\right)+b^{*} \left(\bar{\lambda }\right)m\left(\lambda \right)\right)\left(b^{*} \left(\bar{\lambda }\right)-a^{*} \left(\bar{\lambda }\right)m\left(\lambda \right)\right)^{-1} ,  
\\ \label{add_ineq}
\int_{0}^\beta V_{\lambda }^{*} \left(t\right)W_{\lambda } \left(t\right)V_\lambda\left(t\right)dt\leq 
{\left(b(\bar\lambda)-m^*(\lambda)a(\bar\lambda)\right)^{-1}(\Im m(\lambda))(b^*(\bar{\lambda})-a^*(\bar\lambda)m(\lambda))^{-1}\over\Im\lambda}
\end{gather}
$\forall[0,\beta]\subseteq\bar{\mathcal{I}}$ and therefore
\begin{gather}\label{add_ineq+}
V_{\lambda } \left(t\right)h\in L_{W_{\lambda } \left(t\right)}^{2} \left(\mathcal{I}\right)\forall h\in \mathcal{H}^{n}.
\end{gather} 
Moreover if $a\left(\lambda \right)=a\left(\bar{\lambda }\right),\, b\left(\lambda \right)=b\left(\bar{\lambda }\right)$ as $\Im\lambda\not= 0$ then we can set $D=\mathbb{C}_+$ and
\begin{gather}\label{vlambda}
\int_{0}^\beta V_{\lambda }^{*} \left(t\right)W_{\lambda } \left(t\right)V_\lambda\left(t\right)dt \le \frac{\Im m_{a,b} \left(\lambda \right)}{\Im\lambda },\ \Im\lambda \ne 0.
\end{gather}

\end{proposition}

Proposition \ref{rm21} contains in \cite{KhrabMAG} in less complete form. Therefore we prove it here.

\begin{proof} 
In view of \eqref{13}, \eqref{GEQ__64ad1_}, \eqref{GEQ__65ad1_} $\mathcal{R}_{\lambda } F$ has a representation (\ref{rlambdaf}) where
\begin{gather} 
\label{prop_add1}
V_{\lambda } \left(t\right)=X_{\lambda } \left(t\right)\left(\begin{array}{c} {{I} _{n} } \\ {m\left(\lambda \right)} \end{array}\right)\left(a_{2}^{*} \left(\bar{\lambda }\right)-a_{1}^{*} \left(\bar{\lambda }\right)m\left(\lambda \right)\right)^{-1}.
\end{gather}

Due to Lemma 1.2 from \cite{KhrabArxiv, KhrabMAG} the integrals in \eqref{rlambdaf} converge strongly if the interval of integration is infinite.
 
In view of \cite{GorGor, Khrab5} and the fact that $\mathcal{P}^{*} \left(\bar{\lambda }\right)G\mathcal{P}\left(\lambda \right)=0$ \cite{Khrab5} one has
\begin{gather} 
\label{prop_add2}
a\left(\lambda \right)=\mp i \left(u\left(\lambda \right)+I_{n} \right)S\left(\lambda \right),\, b\left(\lambda \right)=\left(u\left(\lambda \right)-I_{n} \right)S\left(\lambda \right),\, \, \lambda \in {\mathbb C}_{\pm }
\end{gather}
where $u\left(\lambda \right)=u^{*} \left(\bar{\lambda }\right)\in B(\mathcal{H}^n)$ is some contraction, $S\left(\lambda \right),S^{-1} \left(\lambda \right)\in B\left(\mathcal{ H}^{n} \right);\, u\left(\lambda \right),\, S\left(\lambda \right)$ analytically depend on $\lambda \in \mathbb{C}\setminus\mathbb{R}^1$.

In view of \eqref{prop_add2}
\begin{gather}\label{prop_add2+}
K\left(\lambda \right)=-4S^{*} \left(\bar{\lambda }\right)u\left(\lambda \right)S\left(\lambda \right)
\end{gather}
and so $K^{-1} \left(\lambda \right)\in B\left(\mathcal{H}^n\right),\, \lambda \in {D}\bigcup {D}_{1} $ since $u^{-1} \left(\lambda \right)=-2iS\left(\lambda \right)\left(a\left(\lambda \right)-ib\left(\lambda \right)\right)^{-1} \in B\left(\mathcal{H}^{n} \right)\, \lambda \in {D}\cup D_1$.

Using (\ref{prop_add2}), (\ref{prop_add2+}) it can be directly shown that initial conditions in point $t=0$ for solutions $V_{\lambda } \left(t\right)$ \eqref{GEQ__66ad1_} and $V_{\lambda } $ (\ref{prop_add1}) coincides and that $m_{a,b} \left(\lambda \right)=m_{a,b}^{*} \left(\bar{\lambda }\right)$. So (\ref{rlambdaf})-(\ref{GEQ__68ad1_}) is proved.

In view of \cite[p. 450]{Khrab5} one has $\forall [0,\beta]\subseteq\bar{\mathcal{I}}$
\begin{gather}
\label{prop_add3}
\mathcal{P}^{*} \left(\lambda \right)\Delta _{\lambda } \left(0,\, \beta\right)\mathcal{P}\left(\lambda \right)\le \frac{1}{2{\Im}\lambda } \mathcal{P}^{*} \left(\lambda \right)G\mathcal{P}\left(\lambda \right).
\end{gather}

Now inequality (\ref{add_ineq}) (and therefore (\ref{add_ineq+})) follows from (\ref{prop_add3}) in view of (\ref{GEQ__64ad1_}), (\ref{GEQ__67ad1_}), (\ref{prop_add1}).

If $a\left(\lambda \right)=a\left(\bar{\lambda }\right),\, b\left(\lambda \right)=b\left(\bar{\lambda }\right)$ as $\Im\lambda\not= 0$, then operator $u\left(\lambda \right)$ is unitary and independent in $\lambda $ (cf. \cite{Khrab5}). Now in formulae \eqref{GEQ__67ad1_}, \eqref{GEQ__68ad1_} and the right-hand-side of \eqref{add_ineq} we substitute $a\left(\bar{\lambda }\right)$, $b\left(\bar{\lambda }\right)$ by $a\left({\lambda }\right)$, $b\left({\lambda }\right)$ and by direct calculations and with the help of (\ref{prop_add2}) we prove that right hand sides of inequalities (\ref{add_ineq}), (\ref{vlambda}) coincides. The proposition is proved.
\end{proof}

For an arbitrary Nevanlinna pair $\left\{-a\left(\lambda \right),\, b\left(\lambda \right)\right\}$ Weyl solution $V_{\lambda } \left(t\right)$ ({\ref{GEQ__66ad1_}}) does not satisfy in general to inequality (\ref{vlambda}) for $\lambda \in {D}\bigcup {D}_{1} $, and corresponding Weyl function $m_{a,\, b} \left(\lambda \right)$ (\ref{GEQ__68ad1_})  does not satisfy  the condition
\begin{gather}
 \label{mab}
\frac{{\Im}m_{a,b} \left(\lambda \right)}{{\Im}\lambda } \ge 0,\, \, \lambda \in {D}\bigcup {D}_1.
\end{gather}

But if we choose pair $\left\{-a\left(\lambda \right),\, b\left(\lambda \right)\right\}$ "in canonical way" then corresponding Weyl solution $V_{\lambda } $ ({\ref{GEQ__66ad1_}}) satisfies (\ref{vlambda}).

Namely let $v\left(\lambda \right)\in B\left(\mathcal{H}^{n} \right)$ is a contraction analically depending on $\lambda $ in domain ${D}\subseteq {\mathbb C}_{+} $ and let $v^{-1} \left(\lambda \right)\in B\left(\mathcal{ H}^{n} \right),\, \lambda \in {D}$. Let us consider the following pair $\left\{a\left(\lambda \right),\, b\left(\lambda \right)\right\}$, where
\begin{gather}\label{alambda}
a\left(\lambda \right)=-i\left(v\left(\lambda \right)+I_{n} \right),\, b\left(\lambda \right)=v\left(\lambda \right)-I_{n},\ \lambda \in {D}.
\end{gather}

Let us extend the pair $\left\{a\left(\lambda \right),b\left(\lambda \right)\right\}$ (\ref{alambda}) to the domain ${D}_{1} $ which is symmetric to ${D}$ with the respect to real axis in the following way
\begin{gather}\label{barlambda}
v\left(\lambda \right)=\left(v^{*} \left(\bar{\lambda} \right)\right)^{-1} ,\, \, \, \lambda \in {D}_1 
\end{gather}
(and therefore $v^*\left(\bar{\lambda }\right)$ is stretching as $\lambda \in {D}$). As a result we obtain the pair of \eqref{prop_add2} type with $D$ (respectively $D_1$) instead of $\mathbb{C}_+$ (respectively $\mathbb{C}_-$) and
$$u(\lambda)=\begin{cases}v(\lambda),&\lambda\in D\\v^*(\bar{\lambda}),&\lambda\in D_1\end{cases},\quad S(\lambda)=\begin{cases}I_n,&\lambda\in D\\-v(\lambda),&\lambda\in D_1\end{cases}.$$
Therefore if $\lambda \in {D}\bigcup {D}_{1} $ then for pair $\left\{a\left(\lambda \right),b\left(\lambda \right)\right\}$ \eqref{alambda}, \eqref{barlambda} the projections (\ref{GEQ__64ad1_}), (\ref{GEQ__65ad1_}) exist and therefore for operator $M\left(\lambda \right)$ (\ref{13}), (\ref{GEQ__64ad1_}) condition (\ref{GEQ__47++_}) holds and is separated.

\begin{lemma}\label{lm11}
The operator Weyl function $m_{a,b} \left(\lambda \right)$ ({\ref{GEQ__68ad1_}}) corresponding to the pair $\left\{a\left(\lambda \right),\, b\left(\lambda \right)\right\}$ (\ref{alambda}), (\ref{barlambda}) satisfies for any $h\in \mathcal{H}^{n} $ the identity 
\begin{gather}\label{e_lm11}
\Im\left( m_{a,b} \left(\lambda \right)h,h\right)=\frac{1}{4} \left\| \sqrt{v\left(\bar{\lambda }\right)v^{*} \left(\bar{\lambda }\right)-I_{n} } \left(I_{n} -im\left(\lambda \right)\right)g\right\| ^{2} +\Im\left(\, m\left(\lambda \right)g,g\right),\, \, \lambda \in {D},
\end{gather}
where $g=\left(\left(I_{n} -v^{*} \left(\bar{\lambda }\right)\right)+i\left(I_{n} +v^{*} \left(\bar{\lambda }\right)m\left(\lambda \right)\right)\right)^{-1} h$, $(\dots)^{-1}\in B(\mathcal{H}^n)$.

\end{lemma} 
 
\begin{proof} The Weyl function $m_{a,b} \left(\lambda \right)$ ({\ref{GEQ__68ad1_}}), (\ref{alambda}), (\ref{barlambda})  is equal to
\begin{gather}\label{mab1} m_{a,b} (\lambda )=\frac{-1}{4} \left(i
(I_{n} +v^{*} (\bar{\lambda }))-(I_{n} -v^{*} (\bar{\lambda }))m(\lambda )\right)\left(I_{n} -v^{*} (\bar{\lambda })+i(I_{n} +v^{*} (\bar{\lambda }))m(\lambda )\right)^{-1} ,\, \lambda \in {D}, 
\end{gather}
where $(\dots)^{-1}\in B(\mathcal{H}^n)$ in view of \eqref{13+}, \eqref{alambda}, \eqref{barlambda}.

Now identity (\ref{e_lm11}) follows from (\ref{mab1}) by direct calculation.
\end{proof}

In view of the fact that $m_{a,b}\left(\bar{\lambda }\right)=m_{a,b}^{*} \left(\lambda \right)$, inequality  (\ref{add_ineq}), Lemma \ref{lm11}, condition (\ref{star8}) and formula (\ref{prop_add1}) the following theorem is valid.

\begin{theorem}\label{th12}
The solution $V_{\lambda } \left(t\right)$ (\ref{GEQ__66ad1_})-(\ref{GEQ__68ad1_}), (\ref{alambda}), (\ref{barlambda}) satisfies inequality (\ref{vlambda}) for $\lambda \in {D}\bigcup {D}_{1} $ (and therefore $m_{a,b} \left(\lambda \right)$ (\ref{GEQ__68ad1_}), (\ref{alambda}), (\ref{barlambda}) satisfies inequality (\ref{mab}) with $"\gg"$ instead of $"\geq"$).
\end{theorem}
 
\begin{lemma}\label{lm12} Let $v\left(\lambda \right)\in B\left(\mathcal{ H}^{n} \right)$ be a contraction analytically depending on $\lambda \in {\mathbb C}_{+} $. Let limit points of the set $S=\left\{\lambda \in {\mathbb C}_{+} :\, v^{-1} \left(\lambda \right)\notin B\left(\mathcal{H}^{n} \right)\right\}$ that belong to ${\mathbb C}_{+} $ be isolated. Let ${D}={\mathbb C}_{+} \backslash S$. Let us consider operator-function $m_{ab} \left(\lambda \right)$ (\ref{GEQ__68ad1_}), (\ref{alambda}), (\ref{barlambda}) as $\lambda \in {D}\bigcup { D}_{1} $. Then the points of the set $S$ are the removable singular points of this function. If $m_{a,b} \left(\lambda \right)$ is extended on the set $S$ in a proper way then we obtain the Nevanlinna operator-function $m_{a.b}(\lambda)=m_{a,b}^*(\bar\lambda)$.
\end{lemma}

\begin{proof} 
Let $\lambda _{0} \in S$ be not the limit point of $S$. Then $\lambda _{0} $ is a removable singular point of scalar function $\left(m_{a,b} \left(\lambda \right)f,f\right)\, \forall f\in \mathcal{H}^{n} $ in view of (\ref{e_lm11}). Hence $\exists\, m_{0}\in B\left(\mathcal{H}^{n} \right):\, \mathop{\lim }\limits_{\lambda \to \lambda _{0} } \left(m_{a,b} \left(\lambda \right)f,f\right)=\left(m_{0} f,f\right)\forall f\in \mathcal{H}^{n} $ in view of principle of uniform boundedness \cite[p. 164]{Halmos}, \cite[p. 322]{Kato}. If we define $m_{a,b} \left(\lambda \right)$ in point $\lambda _{0} $ as $m_{a,b} \left(\lambda _{0} \right)=m_{0} $ then we obtain the operator-function which is analytic in point $\lambda _{0} $ in view of \cite[p. 195]{Kato}. The analicity of $m_{a,b} \left(\lambda \right)$ in limit points of ${S}$ belonging to $\mathbb{C}_+$ is proved analogously.
\end{proof} 

\begin{corollary}\label{cor11}
Let the construction $v(\lambda)\in B(\mathcal{H}^n)$ satisfy condition of Lemma \ref{lm12}. Then corresponding solution $V_\lambda(t)$ (\ref{GEQ__66ad1_})-(\ref{GEQ__68ad1_}), (\ref{alambda}), (\ref{barlambda}) satisfies inequality (\ref{vlambda}) ($V_\lambda(t)\overset{def}{=} (\ref{prop_add1}),\ \lambda\notin D\cup D_1$).
\end{corollary}

For the construction of solutions of Weyl type and descriptions of Weyl function in various situation see  \cite{0,Khrab5} and references in \cite{0}.

}

We consider in the separable Hilbert space $\mathcal{H}$
differential expression $l_\lambda\left[y\right]$ of order $r>0$ with coefficients from
$B\left(\mathcal{H}\right)$. This
expression is presented in the divergent form, namely
\begin{equation} \label{GEQ__51+_}
l_\lambda\left[y\right]=\sum\limits _{k=0}^{r}i^{k} l_{k}(\lambda)\left[y\right] ,
\end{equation}
where $l_{2j}(\lambda)=D^{j} p_{j} \left(t,\lambda\right)D^{j} $, $l_{2j-1}(\lambda)
=\frac{1}{2} D^{j-1} \left\{Dq_{j} \left(t,\lambda\right)+s_{j}
\left(t,\lambda\right)D\right\}D^{j-1}$, $D={d\over dt}$. 

Let $-l_{\lambda } $ depend on $\lambda $ in Nevanlinna
manner. Namely, from now on the following condition holds:

(\textbf{B}) The set $\mathcal{B}\supseteq {\mathbb{C}\setminus
\mathbb{R}}^1 $ exists, any its points have a neighbourhood
independent on $t\in \bar{\mathcal{I}}$, in this neighbourhood
coefficients $p_{j} =p_{j} \left(t,\lambda \right),\, \, q_{j}
=q_{j} \left(t,\lambda \right),\, \, s=s_{j} \left(t,\lambda
\right)$ of the expression $l_{\lambda } $ are analytic $\forall
t\in \bar{\mathcal{I}}$; $\forall \lambda \in \mathcal{B}{\rm ,}\,
\, p_{j} \left(t,\lambda \right)$, $q_{j} \left(t,\lambda
\right)$, $s_{j} \left(t,\lambda \right)\in C^{j}
\left(\bar{\mathcal{I}},B\left(\mathcal{H}\right)\right)$ and
\begin{equation} \label{GEQ__48_}
p_{n}^{-1} \left(t,\lambda \right)\in
B\left(\mathcal{H}\right)\left(r=2n\right),\, \left(q_{n+1}
\left(t,\lambda \right)+s_{n+1} \left(t,\lambda
\right)\right)^{-1} \in B\left(\mathcal{H}\right)\,
\left(r=2n+1\right),\ t\in\bar{\mathcal{I}};
\end{equation}
these coefficients satisfy the following conditions
\begin{gather} \label{GEQ__49_}
p_{j} \left(t,\lambda \right)=p_{j}^{*} \left(t,\bar{\lambda
}\right),\, q_{j} \left(t,\lambda \right)=s_{j}^{*}
\left(t,\bar{\lambda }\right),\ \lambda \in\mathcal{B}\ (\Longleftrightarrow\ l_{\lambda }
=l_{\bar{\lambda }}^{*},\ \lambda\in\mathcal{B}),
\end{gather}
\begin{multline} \label{GEQ__50_}
\forall h_{0} ,\ldots ,h_{\left[\frac{r+1}{2} \right]} \in
\mathcal{H}:\\ \frac{\Im \left(\sum\limits
_{j=0}^{\left[r/2\right]}\left(p_{j} \left(t,\lambda \right)h_{j}
,h_{j} \right) +\frac{i}{2} \sum\limits
_{j=1}^{\left[\frac{r+1}{2} \right]}\left\{\left(s_{j}
\left(t,\lambda \right)h_{j} ,h_{j-1} \right)-\left(q_{j}
\left(t,\lambda \right)h_{j-1} ,h_{j} \right)\right\} \right)}{\Im
\lambda } \le 0,\\ t\in \bar{\mathcal{I}},\, \, \, \Im \lambda \ne
0.
\end{multline}

Therefore the order of expression $\Im l_{\lambda } $ is even and
therefore if $r=2n+1$ is odd, then $q_{m+1} ,\, s_{m+1} $ are
independent on $\lambda $ and $s_{n+1} =q_{n+1}^{*} $.

Condition \eqref{GEQ__50_} is equivalent to the condition:
${\left(\Im l_{\lambda } \right)\left\{f,f\right\}
\mathord{\left/{\vphantom{\left(\Im l_{\lambda }
\right)\left\{f,f\right\} \Im \lambda
}}\right.\kern-\nulldelimiterspace} \Im \lambda } \le 0,\, \, t\in
\bar{\mathcal{I}},\, \, \Im \lambda \ne 0$. Here for differential expression 
$L[y]=\sum\limits _{k=0}^{R}i^{k} L_{k}[y]
$ with sufficiently smooth coefficients from $B(\mathcal{H})$, where $L_{2j} =D^{j} P_{j} \left(t\right)D^{j}$,\ $L_{2j-1}
=\frac{1}{2} D^{j-1} \left\{DQ_{j} \left(t\right)+S_{j}
\left(t\right)\, D\right\}D^{j-1} $, we denote by 
\begin{multline} \label{GEQ__28_}
L\left\{f,g\right\}=\sum\limits _{j=0}^{\left[{R
\mathord{\left/{\vphantom{R 2}}\right.\kern-\nulldelimiterspace}
2} \right]}\left(P_{j} \left(t\right)f^{\left(j\right)}
\left(t\right),\, g^{\left(j\right)} \left(t\right)\right)+\\+
\frac{i}{2} \sum\limits _{j=1}^{\left[\frac{R+1}{2}
\right]}\left(S_{j} \left(t\right)f^{\left(j\right)}
\left(t\right),\, g^{\left(j-1\right)}
\left(t\right)\right)-\left(Q_{j}
\left(t\right)f^{\left(j-1\right)}
\left(t\right),g^{\left(j\right)} \left(t\right)\right)
\end{multline}
the bilinear form which corresponds to subintegral expression of the Dirichlet integral for 
expression $L[y]$.

Let $m[y]$ be the same as $l_\lambda[y]$ differential expression
of even order $s\le r$ with operator coefficients $\tilde{p}_{j}
\left(t\right)=\tilde{p}_{j}^{*} \left(t\right),\, \,
\tilde{q}_{j} \left(t\right),\, \tilde{s}_{j}
\left(t\right)=\tilde{q}_{j}^{*} \left(t\right)\in C^j(\bar{\mathcal{I}},B(\mathcal{H}))$ that are
independent on $\lambda $. Let
\begin{multline}
\label{GEQ__52_} \forall h_{0} ,\, \ldots ,\,
h_{\left[\frac{r+1}{2} \right]} \in \mathcal{H}:\, \, 0\le
\sum\limits _{j=0}^{{s \mathord{\left/{\vphantom{s
2}}\right.\kern-\nulldelimiterspace} 2} }\left(\tilde{p}_{j}
\left(t\right)h_{j} ,h_{j} \right) +{\Im}\sum\limits _{j=1}^{{s
\mathord{\left/{\vphantom{s 2}}\right.\kern-\nulldelimiterspace}
2} }\left(\tilde{q}_{j}
\left(t\right)h_{j-1} ,\, h_{j} \right) \le\\
 \le -\frac{\Im \left(\sum\limits
_{j=0}^{\left[{r \mathord{\left/{\vphantom{r
2}}\right.\kern-\nulldelimiterspace} 2} \right]}\left(p_{j}
\left(t,\lambda \right)h_{j} ,h_{j} \right) +\frac{i}{2}
\sum\limits _{j=1}^{\left[\frac{r+1}{2} \right]}\left(\left(s_{j}
\left(t,\lambda \right)h_{j} ,h_{j-1} \right)-\left(q_{j}
\left(t,\lambda \right)h_{j-1} ,h_{j} \right)\right) \right)}{\Im
\lambda },\\ t\in \bar{\mathcal{I}},\, \, \, \Im \lambda \ne 0.
\end{multline}
Condition \eqref{GEQ__52_} is equivalent to the condition:
$0\le m\left\{f,f\right\}\le -({\Im l_{\lambda })\left\{f,f\right\}/\Im \lambda } $, $t\in
\bar{\mathcal{I}}$, $\Im \lambda \ne 0$.

In the case of even $r=2n\ge s$  we denote
\begin{gather}\label{GEQ__6_}
Q\left(t,l_\lambda\right)={J}/{i} ,\, \, \,
S\left(t,l_\lambda\right)=Q\left(t,l_\lambda\right), \\ \label{GEQ__7_}
H\left(t,\, l_\lambda\right)=\left\| h_{\alpha \beta } \right\| _{\alpha
,\, \beta =1}^{2} ,\, \, h_{\alpha \beta } \in
B\left(\mathcal{H}^n \right),
\end{gather}
where $h_{11} $ is a three diagonal operator matrix whose
elements under the main diagonal are equal to $\left(\frac{i}{2}
q_{1} ,\, \ldots ,\, \frac{i}{2} q_{n-1} \right)$, the elements
over the main diagonal are equal to $\left(-\frac{i}{2} s_{1} ,\,
\, \ldots ,\, \, -\frac{i}{2} s_{n-1} \right)$, the elements on
the main diagonal are equal to $\left(-p_{0} ,\, \, \ldots ,\, \,
-p_{n-2} ,\, \, \frac{1}{4} s_{n} p_{n}^{-1} q_{n} -p_{n-1}
\right)$; $h_{12} $ is an operator matrix with the identity
operators $I_{1} $ under the main diagonal, the elements on the
main diagonal are equal to $\left(0,\, \, \ldots ,\, \, 0,\, \,
-\frac{i}{2} s_{n} p_{n}^{-1} \right)$, the rest elements are
equal to zero; $h_{21} $ is an operator matrix with identity
operators $I_{1} $ over the main diagonal, the elements on the
main diagonal are equal to $\left(0,\, \, \ldots ,\, \, 0,\, \,
\frac{i}{2} p_{n}^{-1} q_{n} \right)$, the rest elements are equal
to zero; $h_{22} =\mathrm{diag}\left(0,\, \, \ldots ,\, \, 0,\, \,
p_{n}^{-1} \right)$.

Also in this case we denote
\footnote{\label{foot1}$W\left(t,l_\lambda,m\right)$ is given for the case
$s=2n$ . If  $s<2n$ one have set the corresponding elements of
operator matrices $m_{\alpha \beta } $ be equal to zero. In
particular if  $s<2n$ then  $m_{12} =m_{21} =m_{22} =0$  and
therefore $W\left(t,l_\lambda,m\right)=\mathrm{diag}\left(m_{11}
,0\right)$ in view of (14) from \cite{KhrabArxiv,KhrabMAG}.}
\begin{equation} \label{GEQ__8_}
W\left(t,\, l_\lambda,\, m\right)=C^{*-1} \left(t,l_\lambda\right)\left\{\left\|
m_{\alpha \beta } \right\| _{\alpha ,\, \beta =1}^{2}
\right\}C^{-1} \left(t,l_\lambda\right), m_{\alpha \beta } \in
B\left(\mathcal{H}^{n} \right),
\end{equation}
where $m_{11} $ is a tree diagonal operator matrix whose elements
under the main diagonal are equal to $\left(-\frac{i}{2}
\tilde{q}_{1} ,\, \ldots ,\, -\frac{i}{2} \tilde{q}_{n-1}
\right)$, the elements over the main diagonal are equal to
$\left(\frac{i}{2} \tilde{s}_{1} ,\, \ldots ,\, \frac{i}{2}
\tilde{s}_{n-1} \right)$, the elements on the main diagonal are
equal to $\left(\tilde{p}_{0} ,\, \ldots ,\, \tilde{p}_{n-1}
\right)$; $m_{12} =\mathrm{diag}\left(0,\, \ldots ,\, 0,\,
\frac{i}{2} \tilde{s}_{n} \right)$, $m_{21}
=\mathrm{diag}\left(0,\, \ldots ,\, 0,\, -\frac{i}{2}
\tilde{q}_{n} \right)$, $m_{22} =\mathrm{diag}\left(0,\, \ldots
,\, 0,\, \tilde{p}_{n} \right)$.

The operator matrix $C\left(t,l_\lambda\right)$ is defined by the condition
\begin{multline}
\label{GEQ__9_} C\left(t,l_\lambda\right)col\left\{f\left(t\right),\,
f'\left(t\right),\, \ldots ,\, f^{\left(n-1\right)}
\left(t\right),\, f^{\left(2n-1\right)} \left(t\right),\, \ldots
,\, f^{\left(n\right)} \left(t\right)\right\}=\\=
 col\, \left\{f^{\left[0\right]}
\left(t|l_\lambda\right),\, f^{\left[1\right]} \left(t|l_\lambda\right),\, \ldots
,\, f^{\left[n-1\right]} \left(t|l_\lambda\right),\, f^{\left[2n-1\right]}
\left(t\left|l_\lambda\right. \right),\, \ldots ,\, f^{\left[n\right]}
\left(t\left|l_\lambda\right. \right)\right\},
\end{multline}
where $f^{\left[k\right]} \left(t\left|L\right. \right)$ are
quasi-derivatives of vector-function $f\left(t\right)$ that
correspond to differential expression $L[y]$; $C^{-1}(t,l_\lambda)\in B(\mathcal{H}^r)$, $t\in\bar{\mathcal{I}}$, $\lambda\in\mathcal{B}$ in view of (14) from \cite{KhrabArxiv,KhrabMAG}.

The quasi-derivatives corresponding to $l_\lambda[y]$ are equal (cf.
\cite{RBUpsala}) to
\begin{gather} \label{GEQ__10_}
y^{\left[j\right]} \left(t\left|l_\lambda\right. \right)=y^{\left({
j}\right)} \left(t\right),\, \, \, { j}=0,\, \, ...,\, \,
\left[\frac{r}{2} \right]-1, \\ \label{GEQ__11_}
y^{\left[n\right]} \left(t\left|l_\lambda\right.
\right)=\left\{\begin{array}{l} {p_{n} y^{\left(n\right)}
-\frac{i}{2} q_{n} y^{\left(n-1\right)} ,\, \, r=2n} \\
{-\frac{i}{2} q_{n+1} y^{\left(n\right)} ,\, \, r=2n+1}
\end{array}\right.,
\\\label{GEQ__12_}
y^{\left[r-j\right]} \left(t\left|l_\lambda\right.
\right)=-Dy^{\left[r-j-1\right]} \left(t\left|l_\lambda\right.
\right)+p_{j} y^{\left(j\right)} +\frac{i}{2} \left[s_{j+1}
y^{\left(j+1\right)} -q_{j} y^{\left(j-1\right)} \right],\, j=0,\,
...,\, \left[\frac{r-1}{2} \right],\, q_{0} \equiv 0.
\end{gather}
At that $l_\lambda\left[y\right]=y^{\left[r\right]} \left(t\left|l_\lambda\right.
\right)$. The quasi-derivatires $y^{\left[k\right]}
\left(t\left|m\right. \right)$ corresponding to $m[y]$ are defined in
the same way with even $s$ instead of $r$ and $\tilde{p}_{j},
\tilde{q}_{j} ,\tilde{s}_{j} $ instead of $p_{j} ,q_{j} ,s_{j} $.

In the case of odd $r=2n+1>s$ we denote
\begin{gather} \label{GEQ__14_}
Q\left(t,l_\lambda\right)=\begin{cases}J/i\oplus q_{n+1}\\
q_{1}\end{cases},\quad S\left(t,l_\lambda\right)=\begin{cases} J/i\oplus
s_{n+1},& n>0 \\ s_{1},& n=0\end{cases},
\\
\label{GEQ__15_} H\left(t,\, l_\lambda\right)=\begin{cases}\left\|
h_{\alpha \,\beta } \right\|_{\alpha ,\, \beta =1}^{2},& n>0 \\
p_{0},& n=0
\end{cases},
\end{gather}
where $B\left(\mathcal{H}^{n} \right)\ni h_{11} $ is a
three-diagonal operator matrix whose elements under the main
diagonal are equal to $\left(\frac{i}{2} q_{1} ,\, \ldots ,\,
\frac{i}{2} q_{n-1} \right)$, the elements over the main diagonal
are equal to $\left(-\frac{i}{2} s_{1} ,\, \ldots ,\, -\frac{i}{2}
s_{n-1} \right)$, the elements on the main diagonal are equal to
$\left(-p_{0} ,\, \ldots ,\, -p_{n-1} \right)$, the rest elements
are equal to zero. $B\, \left(\mathcal{ H}^{n+1} ,\, \mathcal{
H}^{n} \right)\ni h_{12} $ is an operator matrix whose elements
with numbers $j,\, j-1$ are equal to $I_{1} ,\, j=2,\, \ldots ,\,
n$, the element with number $n,\, n+1$ is equal to $\frac{1}{2}
s_{n} $, the rest elements are equal to zero. $B\,
\left(\mathcal{H}^{n} ,\, \mathcal{H}^{n+1} \right)\ni h_{21} $ is
an operator matrix whose elements with numbers $j-1,\, j$ are
equal to $I_{1} ,\, j=2,\, \ldots ,\, n$, the element with number
$n+1,\, n$ is equal to $\frac{1}{2} q_{n} $, the rest elements are
equal to zero. $B\, \left(\mathcal{H}^{n+1} \right)\ni h_{22} $ is
an operator matrix whose last row is equal to $\left(0,\, \ldots
,\, 0,\, -iI_{1} ,\, -p_{n} \right)$, last column is equal to
$col\, \left(0,\, \ldots ,\, 0,\, iI_{1} ,\, -p_{n} \right)$, the
rest elements are equal to zero.

Also in this case we denote \footnote{See the previous footnote}
\begin{equation} \label{GEQ__16_}
W\left(t,\, l_\lambda,\, m\right)=\left\| m_{\alpha \beta } \right\| _{\alpha ,\, \beta =1}^{2} ,
\end{equation}
where $m_{11} $ is defined in the same way as $m_{11} $
\eqref{GEQ__8_}. $B\left(\mathcal{H}^{n+1} ,\, \mathcal{H}^{n}
\right)\ni m_{12} $ is an operator matrix whose element with
number $n,\, n+1$ is equal to $-\frac{1}{2} \tilde{s}_{n} $, the
rest elements are equal to zero. $B\, \left(\mathcal{H}^{n} ,\,
\mathcal{H}^{n+1} \right)\ni m_{21} $ is an operator matrix whose
element with number $n+1,\, n$ is equal to $-\frac{1}{2}
\tilde{q}_{n} $, the rest elements are equal to zero. $B\,
\left(\mathcal{H}^{n+1} \right)\ni m_{22} =\mathrm{diag}\left(0,\,
\ldots ,\, 0,\, \tilde{p}_{n} \right)$.

Obviously in view of \eqref{GEQ__49_}, \eqref{GEQ__52_} for $H\left(t,l_\lambda\right)$ \eqref{GEQ__7_},
\eqref{GEQ__15_} and $W\left(t,l_\lambda,m\right)$ 
\eqref{GEQ__8_}, \eqref{GEQ__16_} one has 
\begin{equation} \label{GEQ__17_}
H^{*} \left(t,l_\lambda\right)=H\left(t,l_{\bar\lambda} \right), W^{*} \left(t,l_\lambda,\, m\right)=W\left(t,l_\lambda,\, m \right),\ t\in\bar{\mathcal{I}},\lambda\in\mathcal{B}.
\end{equation}

\begin{lemma}\cite{KhrabMAG} (see also \cite{KhrabArxiv})\label{lm1}
Let the order of $\Im l_\lambda$ is even. Then
\begin{equation} \label{GEQ__18_}
\Im H\left(t,l_\lambda\right)=W\left(t,l_\lambda,\, -\Im l_\lambda\right)=W\left(t,l_{\bar\lambda}
,\, -\Im l_\lambda\right),\ t\in\bar{\mathcal{I}},\lambda\in\mathcal{B}.
\end{equation}
\end{lemma}

For sufficiently smooth vector-function $f\left(t\right)$ by
corresponding capital letter we denote (if $f(t)$ has a subscript
then we add the same subscript to $F$)
\begin{multline} \label{GEQ__23_}
\mathcal{H}^{r} \ni F\left(t,\, l_\lambda,m\right)=\\=\begin{cases}
\left(\sum\limits _{j=0}^{s/2}\oplus f^{\left(j\right)}
\left(t\right) \right)\oplus
0 \oplus ...\oplus 0 ,\quad r=2n,&r=2n+1>1,s<2n\\
\left(\sum\limits _{j=0}^{n-1}\oplus f^{\left(j\right)}
\left(t\right) \right)\oplus 0 \oplus ...\oplus 0 \oplus \,
\left(-if^{\left(n\right)} \left(t\right)\right),&r=2n+1>1,
s=2n\\ f\left(t\right),&r=1 \\ \left(\sum\limits
_{j=0}^{n-1}\oplus f^{\left(j\right)} \left(t\right) \right)\oplus
\left(\sum\limits _{j=1}^{n}\oplus f^{\left[r-j\right]}
\left(t\left|l_\lambda\right. \right) \right),& r=s=2n.\end{cases}
\end{multline}

\begin{theorem}\cite{KhrabMAG} (see also \cite{KhrabArxiv})\label{th1}
Equation \eqref{GEQ__1_} is equivalent to the following first
order system
\begin{equation} \label{GEQ__54_}
\frac{i}{2} \left(\left(Q\left(t,l_{\lambda }
\right)\vec{y}\left(t\right)\right)^{{'} } +Q^{*}
\left(t,l_{\lambda }
\right)\vec{y}\hspace{2px}'\left(t\right)\right)-H\left(t,l_{\lambda
} \right)\vec{y}\left(t\right)=W\left(t,l_{\bar{\lambda }}
,m\right)F\left(t,l_{\bar{\lambda }} ,m\right),
\end{equation}
where $Q\left(t,l_{\lambda } \right),\, H\left(t,l_{\lambda }
\right)$ are defined by \eqref{GEQ__6_}, \eqref{GEQ__7_},
\eqref{GEQ__14_}, \eqref{GEQ__15_} and $W\left(t,l_{\bar{\lambda }} ,m\right)$,
$F\left(t,l_{\bar{\lambda }} ,m\right)$ are defined by
\eqref{GEQ__8_}, \eqref{GEQ__16_}, \eqref{GEQ__23_} with $l_{\bar\lambda } $ instead of
$l_\lambda$.
Namely if $y\left(t\right)$ is a solution of equation
\eqref{GEQ__1_} then
\begin{multline} \label{GEQ__25_}
\vec{y}\left(t\right)=\vec{y}\left(t,\, l_\lambda,\, m,\,
f\right)=\\=\begin{cases}\left(\sum\limits _{j=0}^{n-1}\oplus
y^{\left(j\right)} \left(t\right) \right)\oplus \left(\sum\limits
_{j=1}^{n}\oplus \left(y^{\left[r-j\right]} \left(t\left|l_\lambda\right.
\right)-f^{\left[s-j\right]}
\left(t\left|m\right. \right)\right) \right),& r=2n\\
\left(\sum\limits _{j=0}^{n-1}\oplus y^{\left(j\right)}
\left(t\right) \right)\oplus \left(\sum\limits _{j=1}^{n}\oplus
\left(y^{\left[r-j\right]} \left(t\left|l_\lambda\right.
\right)-f^{\left[s-j\right]} \left(t\left|m\right. \right)\right)
\right)\oplus \left(-iy^{\left(n\right)} \left(t\right)\right),&
r=2n+1>1 \\\quad \text{{\rm (here }}f^{\left[k\right]}
\left(t\left|m\right.
\right)\equiv 0\text{ \rm as }k< \frac{s}{2}\text{\rm)}\\
\\y\left(t\right),& r=1\end{cases}
\end{multline}
is a solution of \eqref{GEQ__54_}. Any solution of
equation \eqref{GEQ__54_}  is equal to \eqref{GEQ__25_}, where
$y\left(t\right)$ is some solution of equation
\eqref{GEQ__1_}.
\end{theorem}

Let us notice that different vector-functions $f(t)$ can generate
different right-hand-sides of equation \eqref{GEQ__54_} but unique right-hand-side of
equation \eqref{GEQ__1_}.

Due to Theorem 1.2 from \cite{KhrabMAG,KhrabArxiv} and Lemma \ref{lm1} we  have $\frac{\Im H\left(t,l_{\lambda } \right)}{\Im
\lambda }=W\left(t,l_{\lambda } ,-\frac{\Im l_{\lambda } }{\Im
\lambda } \right) \ge 0,\, \, t\in \bar{\mathcal{I}},\, \, \, \Im \lambda
\ne 0$  and therefore
$H\left(t,l_{\lambda } \right)$ satisfy condition \textbf{(A)}
with $\mathcal{A}=\mathcal{B}$. Therefore $\forall\mu\in
\mathcal{B}\cap \mathbb{R}^1$ $W(t,l_\mu,-{\Im l_\mu\over\Im\mu
})=\left.{\partial
H(t,l_\lambda)\over\partial\lambda}\right|_{\lambda=\mu}$ is
Bochner locally integrable in uniform operator topology. Here in
view of \eqref{GEQ__7_}, \eqref{GEQ__15_} $\forall \mu \in
\mathcal{B}\bigcap \mathbb{R}^{1}\ \exists \frac{\Im l_{\mu }
}{\Im \mu } \mathop{=}\limits^{def} \frac{\Im l_{\mu +i0} }{\Im
\left(\mu +i0\right)} =\left.\frac{\partial l_{\lambda }
}{\partial \lambda }\right|_{\lambda=\mu} $, where the
coefficients ${\partial p_j(t,\mu)\over\partial\lambda}$,
${\partial q_j(t,\mu)\over\partial\lambda}$, ${\partial
s_j(t,\mu)\over\partial\lambda}$ of expression ${\partial l_{\mu }
\mathord{\left/{\vphantom{\partial l_{\mu }
\partial \mu }}\right.\kern-\nulldelimiterspace} \partial \mu } $
are Bochner locally integrable in the uniform operator topology.

Also in view of Theorem 1.2 from \cite{KhrabMAG,KhrabArxiv} and Lemma \ref{lm1} one has
\begin{equation} \label{GEQ__53_}
0\le W\left(t,l_{\lambda } ,m\right)\le W\left(t,l_{\lambda }
,-\frac{\Im l_{\lambda } }{\Im \lambda } \right)=\frac{\Im
H\left(t,l_{\lambda } \right)}{\Im \lambda } \quad t\in
\bar{\mathcal{I}},\, \, \, \Im \lambda \ne 0
\end{equation}

Let us consider in $\mathcal{H}_1=\mathcal{H}^r$ the equation
\begin{equation} \label{GEQ__51_}
\frac{i}{2} \left(\left(Q\left(t,l_{\lambda }
\right)\vec{y}\left(t\right)\right)^{{'} } +Q^{*}
\left(t,l_{\lambda }
\right)\vec{y}\hspace{2px}'\left(t\right)\right)-H\left(t,l_{\lambda
} \right)\vec{y}\left(t\right)=W\left(t,l_{\lambda } -\frac{\Im
l_{\lambda } }{\Im \lambda } \right)F\left(t\right).
\end{equation}
This equation is an equation of \eqref{GEQ__46_} type due to
\eqref{GEQ__17_}, \eqref{GEQ__53_}. Equation
\eqref{GEQ__5_} is equivalent to equation \eqref{GEQ__51_}
with $F\left(t\right)=F\left(t,l_{\bar{\lambda }} ,-\frac{\Im
l_{\lambda } }{\Im \lambda } \right)$ due to Theorem \ref{th1} and \eqref{GEQ__18_}.

\begin{definition}\cite{KhrabMAG,KhrabArxiv} Every characteristic operator of equation
\eqref{GEQ__51_} corresponding to the equation
\eqref{GEQ__5_} is said to be a characteristic operator of
equation \eqref{GEQ__5_} on $\mathcal{I}$.
\end{definition}

In some cases we will suppose additionally that
\begin{multline}\exists \lambda _{0} \in \mathcal{B};\, \, \alpha
,\beta \in \bar{\mathcal{I}},\, \, 0\in \left[\alpha ,\beta
\right]\text{, the number }\delta >0:\\
\label{GEQ__55_}
-\int _{\alpha }^{\beta }\left(\frac{\Im l_{\lambda _{0} } }{\Im
\lambda _{0} } \right)\left\{y\left(t,\lambda _{0}
\right),y\left(t,\lambda _{0} \right)\right\} \, dt\ge \delta
\left\| P\vec{y}\left(0,l_{\lambda _{0} } ,m,0\right)\right\| ^{2}
\end{multline}
for any solution $y\left(t,\lambda _{0} \right)$ of
\eqref{GEQ__1_} as $\lambda =\lambda _{0} ,\, \, f=0$, where $P\in B(\mathcal{H}^r)$ is the orthoprojection onto subspace $N^\perp$ which corresponds to equation (\ref{GEQ__51_}). In view
of Theorem 1.2 from \cite{KhrabMAG,KhrabArxiv} this condition is equivalent to the fact that
for the equation \eqref{GEQ__51_}
\begin{gather*}\exists \lambda _{0} \in \mathcal{A}=\mathcal{B};\, \, \alpha
,\beta \in \bar{\mathcal{I}},\, \, 0\in \left[\alpha ,\beta
\right]\text{, the number }\delta >0:
\ \left(\Delta _{\lambda _{0} } \left(\alpha ,\beta
\right)g,g\right)\ge \delta \left\| Pg\right\| ^{2} ,\quad g\in
\mathcal{H}^{r} .
\end{gather*}

Therefore in view of \cite{Khrab5} the fulfillment of
\eqref{GEQ__55_} implies its fulfillment with $\delta
\left(\lambda \right)>0$ instead of $\delta $ for all $\lambda \in
\mathcal{B}$.

Let us notice what in view of \eqref{GEQ__52_} $l_{\lambda } $
can be a represented in form \eqref{GEQ__2_} where
\begin{equation} \label{GEQ__561_}
l=\Re l_{i} , n_{\lambda } =l_{\lambda } -l-\lambda m; {\Im
n_{\lambda } \left\{f,f\right\}
\mathord{\left/{\vphantom{n_{\lambda } \left\{f,f\right\}
\mathcal{I}\lambda \ge 0}}\right.\kern-\nulldelimiterspace}
\Im\lambda \ge 0} ,t\in \bar{\mathcal{I}},\, {\Im}\lambda \ne 0.
\end{equation}

From now on we suppose that $l_{\lambda } $ has a representation
 \eqref{GEQ__2_}, \eqref{GEQ__561_} and therefore the
order of $n_\lambda$ is even.

We consider pre-Hilbert spaces $\mathop{H}\limits^{\circ } $ and
$H$ of vector-functions $y\left(t\right)\in C_{0}^{s}
\left(\bar{\mathcal{I}},\mathcal{H}\right)$ and
$y\left(t\right)\in C^{s}
\left(\bar{\mathcal{I}},\mathcal{H}\right),\,
m\left[y\left(t\right),\, y\left(t\right)\right]<\infty $
correspondingly with a scalar product
\[\left(f\left(t\right),\, g\left(t\right)\right)_{m} =m\left[f\left(t\right),\, g\left(t\right)\right],\]
where  
\begin{gather}
m\left[f,\, g\right]=\int\limits_{\mathcal{I}}m\left\{f,\, g\right\}dt,
\end{gather}
Here $m\left\{f,\, g\right\}$ is defined by \eqref{GEQ__28_}
with expression $m[y]$ from condition \eqref{GEQ__52_} instead of
$L[y]$. Namely, $$m\left\{f,\, g\right\}=\sum\limits _{j=0}^{s/2} (\tilde{p}_{j}
\left(t\right)f^{(j)}(t) ,g^{(j)}(t) ) +{i\over 2}\sum\limits _{j=1}^{s/2}
\left((\tilde{q}^*_{j}
\left(t\right)f^{(j)}(t) ,\, g^{(j-1)}(t) )-(\tilde{q}_{j}
\left(t\right)f^{(j-1)}(t) ,\, g^{(j)}(t) )\right).$$

By $\mathop{L_{m}^{2} }\limits^{\circ } \left(\mathcal{I}\right)$
and $L_{m}^{2} \left(\mathcal{I}\right)$ we denote the completions
of spaces $\mathop{H}\limits^{\circ } $ and $H$ in the norm
$\left\| \, \bullet \, \right\| _{m} =\sqrt{\left(\, \bullet ,\,
\bullet \right)_{m} } $ correspondingly. By
$\mathop{P}\limits^{\circ } $ we denote the orthoprojection in
$\mathop{L_{m}^{2} }\limits^{} \left(\mathcal{I}\right)$ onto
$\mathop{L_{m}^{2} }\limits^{\circ } \left(\mathcal{I}\right)$.

\begin{theorem} \cite{KhrabMAG} (see also \cite{KhrabArxiv})\label{th4}
Let $M\left(\lambda \right)$ be a characteristic operator of equation
\eqref{GEQ__5_}, for which the condition \eqref{GEQ__55_}
with $P=I_{r} $ holds if $\mathcal{I}$ is infinite. Let
$\Im\lambda\not= 0$, $f\left(t\right)\in H$ and
\begin{multline} \label{GEQ__62_}
col\left\{y_{j} \left(t,\lambda ,f\right)\right\}=\\=\int
_{\mathcal{I}}X_{\lambda } \left(t\right) \left\{M\left(\lambda
\right)-\frac{1}{2} sgn\left(s-t\right)\left(iG\right)^{-1}
\right\}X_{\bar{\lambda }}^{*}
\left(s\right)W\left(s,l_{\bar{\lambda }}
,m\right)F\left(s,l_{\bar{\lambda }} ,m\right)\, ds,\, y_{j} \in
\mathcal{H}
\end{multline}
be a solution of equation \eqref{GEQ__54_}, that corresponds
to equation \eqref{GEQ__1_}, where $X_{\lambda }
\left(t\right)$ is the operator solution of homogeneons equation
\eqref{GEQ__54_} such that $X_{\lambda } \left(0\right)=I_{r}
;\, \, G=\Re Q\left(0,l_{\lambda } \right)$ (if $\mathcal{I}$ is
infinite integral \eqref{GEQ__62_} converges strongly). Then
the first component of vector function \eqref{GEQ__62_} is a
solution of equation \eqref{GEQ__1_}. It defines densely
defined in $L_{m}^{2} \left(\mathcal{I}\right)$
integro-differential operator
\begin{equation} \label{GEQ__63_}
R\left(\lambda \right)f=y_{1} \left(t,\lambda ,f\right),\quad f\in H
\end{equation}
which has the following properties after closing

\noindent 1${}^\circ$
\begin{gather}\label{GEQ__64_}
R^{*} \left(\lambda \right)=R\left(\bar{\lambda }\right),\, \, \,
{\Im}\lambda \ne 0
\end{gather}

\noindent 2${}^\circ$
\begin{gather}\label{GEQ__65_}
R\left(\lambda \right)\text{ is holomorphic on
}\mathbb{C}\setminus \mathbb{R}^1
\end{gather}

\noindent 3${}^\circ$
\begin{gather}\label{GEQ__66_}
\left\| R\left(\lambda \right)f\right\| _{L_{m}^{2}
\left(\mathcal{I}\right)}^{2} \le \frac{\Im \left(R\left(\lambda
\right)f,f\right)_{L_{m}^{2} \left(\mathcal{I}\right)} }{\Im
\lambda } ,\, \, \, {\Im}\lambda \ne 0,\, \, \, f\in L_{m}^{2}
\left(\mathcal{I}\right)
\end{gather}
\end{theorem}

Let us notice that the definition of the operator $R\left(\lambda
\right)$ is correct. Indeed if $f\left(t\right)\in H$,
$m\left[f,f\right]=0$, then $R\left(\lambda \right)f\equiv 0$
since $W\left(t,l_{\bar{\lambda }}
,m\right)F\left(t,l_{\bar{\lambda }},m \right)\equiv 0$ due to
\eqref{GEQ__53_} and Theorem 1.2 from \cite{KhrabMAG,KhrabArxiv}.

Also let us notice that if $L^2_m(\mathcal{I})=\overset{\circ\ }{L^2_m}(\mathcal{I})$ then Theorem \ref{th4} is valid with $f(t)\in \overset{\circ}H$ instead of $f(t)\in H$ and without condition \eqref{GEQ__55_} with $P=I_r$ if $\mathcal{I}$ is infinite.

The resolvent $R(\lambda)$ can be represented in another forms. (In \eqref{prop12formula}, \eqref{prop13formula} integrals converge strongly if the interval of integration is infinite.)

\begin{proposition}\cite{KhrabMAG}\label{rm31} Let us represent characteristic operator $M\left(\lambda \right)$ from Theorem \ref{th4} in the form \eqref{13}. Then $R(\lambda)f$  \eqref{GEQ__63_} can be represented in the form
\begin{multline}\label{prop12formula}
R\left(\lambda \right)f=\int _{a}^{t}\sum _{j=1}^{r}y_{j} \left(t,\lambda \right) \sum _{k=0}^{{s\mathord{\left/ {\vphantom {s 2}} \right. \kern-\nulldelimiterspace} 2} }\left(x_{j}^{\left(k\right)} \left(s,\bar{\lambda }\right)\right)^{*}  \mathrm{m}_{k} \left[f\left(s\right)\right]ds+\\+
\int _{t}^{b}\sum _{j=1}^{r}x_{j} \left(t,\lambda \right) \sum _{k=0}^{{s\mathord{\left/ {\vphantom {s 2}} \right. \kern-\nulldelimiterspace} 2} }\left(y_{j}^{\left(k\right)} \left(s,\bar{\lambda }\right)\right)^{*}  \mathrm{m}_{k} \left[f\left(s\right)\right]ds
\end{multline} 
where $x_{j} \left(t,\lambda \right),y_{j} \left(t,\lambda \right)\in B\left(\mathcal{H}\right)$ are operator solutions of equation \eqref{GEQ__1_} as $f=0$, such that $\left(x_{1} \left(t,\lambda \right),\, \ldots ,x_{r} \left(t,\lambda \right)\right)$ is the first row $\left[X_{\lambda } \left(t\right)\right]_{1}\in B(\mathcal{H}^r,\mathcal{H}) $ of operator matrix $X_{\lambda } \left(t\right),\, \left(y_{1} \left(t,\lambda \right),\, \ldots ,y_{r} \left(t,\lambda \right)\right)=\left[X_{\lambda } \left(t\right)\right]_{1} \mathcal{P}\left(\lambda \right)\left(iG\right)^{-1} $, 
\begin{gather}\label{mk}
\mathrm{m}_{k} \left[f\left(s\right)\right]=\tilde{p}_{k} \left(s\right)f^{\left(k\right)} \left(s\right)+\frac{i}{2} \left(\tilde{q}_{k}^{*} \left(s\right)f^{\left(k+1\right)} \left(s\right)-\tilde{q}_{k} \left(s\right)f^{\left(k-1\right)} \left(s\right)\right)\left(\tilde{q}_{0} \equiv 0,\, \tilde{q}_{\frac{s}{2} +1} \equiv 0\right).
\end{gather}
\end{proposition}

In view of Proposition \ref{rm21}, Corollary \ref{cor11}, Theorem \ref{th4} and also Theorem 1.2 from \cite{KhrabArxiv,KhrabMAG} the following proposition is valid.

\begin{proposition}\cite{KhrabMAG}\label{rm32} Let $r=2n,\ \mathcal{I}=(0,b)$, $b\leq\infty$, condition \eqref{GEQ__55_} hold with $P=I_r$. (Therefore for equation \eqref{GEQ__51_} condition \eqref{star8} holds.) Let for characteristic operator $M\left(\lambda \right)$ of equation \eqref{GEQ__5_} condition \eqref{GEQ__47++_} be separated. (Therefore $M\left(\lambda \right)$ has representation \eqref{13} where characteristic projection $\mathcal{ P}\left(\lambda \right)$ can be represented in the form \eqref{GEQ__64ad1_}, \eqref{GEQ__65ad1_} with the help of some Nevanlinna pair $\{-a(\lambda),b(\lambda)\}$ and some Weyl function $m(\lambda)$ of equation \eqref{GEQ__51_}; this equation with $F(t)=0$ has an operator solutions $U_{\lambda } \left(t\right),\, V_{\lambda } \left(t\right)$ \eqref{GEQ__66ad1_}-\eqref{GEQ__68ad1_}). Let domains ${D,}\, {D}_{1} $ be the same as in Proposition \ref{rm21}. Then $R(\lambda)f$  \eqref{GEQ__63_} for $\lambda \in {D}\bigcup {D}_{1}$ can be represented in the form
\begin{multline}\label{prop13formula}
R\left(\lambda \right)f=\int _{0}^{t}\sum _{j=1}^{n}v_{j} \left(t,\lambda \right) \sum _{k=0}^{s/2}\left(u_{j}^{\left(k\right)} \left(s,\bar{\lambda }\right)\right)^{*}  \mathrm{m}_{k} \left[f\left(s\right)\right]ds +\\ 
+\int _{t}^{b}\sum _{j=1}^{n}u_{j} \left(t,\lambda \right) \sum _{k=0}^{s/2}\left(v_{j}^{\left(k\right)} \left(s,\bar{\lambda }\right)\right)^{*}  \mathrm{m}_{k} \left[f\left(s\right)\right]ds ,  
\end{multline} 
where $u_{j} \left(t,\lambda \right),\, v_{j} \left(t,\lambda \right)\in B\left(\mathcal{H}\right)$ are operator solutions of equation \eqref{GEQ__1_} as $f=0$, such that, $\left(u_{1} \left(t,\lambda \right),\ldots u_{n} \left(t,\lambda \right)\right)=\left[X_{\lambda } \left(t\right)\right]_{1} \left(\begin{array}{c} {a\left(\lambda \right)} \\ {b\left(\lambda \right)} \end{array}\right)$,
\begin{equation} \label{GEQ__115_} 
\left(v_{1} \left(t,\lambda \right),\ldots ,v_{n} \left(t,\lambda \right)\right)=\left[X_{\lambda } \left(t\right)\right]_{1} \left(\begin{array}{c} {b\left(\lambda \right)} \\ {-a\left(\lambda \right)} \end{array}\right)K^{-1} \left(\lambda \right)+\left(u_{1} \left(t,\lambda \right),\ldots ,u_{n} \left(t,\lambda \right)\right)m_{a,b} \left(\lambda \right),  
\end{equation} 
$K\left(\lambda \right),\, m_{a,b} \left(\lambda \right)$ see \eqref{GEQ__67ad1_}, \eqref{GEQ__68ad1_};  
$$\|\left(v_{1} \left(t,\lambda \right),\ldots ,v_{n} \left(t,\lambda \right)\right)h\|^2_{m}\leq {\Im(m(\lambda)g,g)\over \Im\lambda},\ \Im\lambda\not= 0,$$
where $g=(b^*(\bar\lambda)-a^*(\bar\lambda)m(\lambda))^{-1}h$, $h\in\mathcal{H}^n$ and therefore $$\left(v_{1} \left(t,\lambda \right),\ldots ,v_{n} \left(t,\lambda \right)\right)h\in L_{m}^{2} \left(\mathcal{I}\right)\, \, \forall h\in \mathcal{H}^{n}. $$

Moreover if $a\left(\lambda \right)=a\left(\bar{\lambda }\right),\, b\left(\lambda \right)=b\left(\bar{\lambda }\right)$ as $\Im\lambda\not= 0$ then we can set $D=\mathbb{C}_+$ and
\begin{gather}\label{norm}
\left\| \left(v_{1} \left(t,\lambda \right),\ldots ,v_{n} \left(t,\lambda \right)\right)h\right\| _{m}^{2} \le \frac{\Im\left(m_{a,b} \left(\lambda \right)h,h\right)}{\Im\lambda },\ \Im\lambda \ne 0.
\end{gather}

Let contraction $v(\lambda)\in B(\mathcal{H}^n)$ satisfy the conditions of Lemma \ref{lm12} and domains $D$, $D_1$ be the same as in Lemma \ref{lm12}. Then corresponding solution $\left(v_{1} \left(t,\lambda \right),\ldots ,v_{n} \left(t,\lambda \right)\right)$ (\ref{GEQ__115_}),(\ref{GEQ__67ad1_}),(\ref{GEQ__68ad1_}),(\ref{alambda}),(\ref{barlambda}) satisfies inequality \eqref{norm} ($\left(v_{1} \left(t,\lambda \right),\ldots ,v_{n} \left(t,\lambda \right)\right)\overset{def}{=}[V_\lambda(t)]_1$, $\lambda\notin D\cup D_1$, where $[V_\lambda(t)]_1\in B(\mathcal{H}^n,\mathcal{H})$ is an analogue of $[X_\lambda(t)]_1$ for $V_\lambda(t)$ \eqref{prop_add1}).

\end{proposition}

Comparison of Theorem \ref{th4}, Propositions \ref{rm31} , \ref{rm32} with results for various particular cases see in \cite{KhrabMAG}.


\section{Eigenfunction expansions}

It is known \cite{DSnoo1} that \eqref{GEQ__64_} -
\eqref{GEQ__66_} implies \eqref{GEQ__3_}, where
$E_{\mu}\in B\left(L_{m}^{2} \left(\mathcal{I}\right)\right)$,
$E_{\mu }=E_{\mu -0}$,
\begin{equation} \label{GEQ__83_}
0\le E_{\mu _{1} } \le E_{\mu _{2} } \le \mathbf{I} ,\, \, \, \mu
_{1} <\mu _{2} ;\, \, \, E_{-\infty } =0.
\end{equation}
Here $\mathbf{I}$ is the identity operator in $L_{m}^{2}
\left(\mathcal{I}\right)$. We denote $E_{\alpha \beta }
=\frac{1}{2} \left[E_{\beta +0} +E_{\beta } -E_{\alpha +0}
-E_{\alpha } \right]$.

\begin{theorem}\label{th7}
Let $M\left(\lambda \right)$ be the characteristic operator of
equation \eqref{GEQ__5_} (and therefore by
\cite[p.162]{Khrab5} $\Im {P} M\left(\lambda \right){P} /\Im
\lambda \ge 0$ as $\Im \lambda \ne 0$) and $\sigma \left(\mu
\right)=w-\mathop{\lim }\limits_{\varepsilon \downarrow 0}
\frac{1}{\pi } \int _{0}^{\mu }\Im  { P} M\left(\mu +i\varepsilon
\right) {P} d\mu $ be the spectral operator-function that
corresponds to ${P} M\left(\lambda \right){P} $.

Let the condition
\eqref{GEQ__55_} with $P=I_r$ hold if $\mathcal{I}$ is
infinite. Let $E_{\mu } $ be generalized spectral family
\eqref{GEQ__83_} corresponding by \eqref{GEQ__3_} to the
resolvent $R\left(\lambda \right)$ from Theorem \ref{th4} which is
constructed with the help of characteristic operator $M\left(\lambda \right)$. Let $\mathcal{B}^1=\mathcal{B}\cap\mathbb{R}^1$. Then
for any $\left[\alpha ,\beta \right]\subset \mathcal{B}^1$ the
equalities
\begin{equation} \label{GEQ__84_}
\begin{matrix}
{\mathop{P}\limits^{\circ } E_{\alpha ,\beta } f\left(t\right)=
\mathop{P}\limits^{\circ } \int _{\alpha }^{\beta }\left[X_{\mu }
\left(t\right)\right]_{1}  d\sigma \left(\mu \right)\varphi
\left(\mu ,\, f\right),\text{ if }f\left(t\right)\in
\mathop{H}\limits^{\circ },\ \mathcal{I}\text{ is infinite},}
\\ {E_{\alpha ,\beta } f\left(t\right)=\int _{\alpha }^{\beta }\left[X_{\mu } \left(t\right)\right]_{1} d\sigma \left(\mu \right)\varphi \left(\mu ,f\right) ,\, \, if\, \, f\left(t\right)\in H,\, \mathcal{I}\text{ is finite}} \end{matrix}
\end{equation}
are valid in $L_{m}^{2} \left(\mathcal{I}\right)$, where
$\left[X_{\lambda } \left(t\right)\right]_{1} \in
B\left(\mathcal{H}^{r} ,\, \mathcal{H}\right)$ is the first row of
the operator solution $X_{\lambda } \left(t\right)$ of homogeneous
equation \eqref{GEQ__54_} which is written in the matrix form and such
that $X_{\lambda } \left(0\right)=I_{r} $,
\begin{equation} \label{GEQ__85_}
\varphi \left(\mu ,f\right)=\left\{\begin{array}{l} {\int
_{\mathcal{I}}\left(\left[X_{\mu } \left(t\right)\right]_{1}
\right)^{*} m\left[f\right]dt\text{ if }f\left(t\right)\in
\mathop{H}\limits^{\circ } } \\ {\int
_{\mathcal{I}}\left(\left[X_{\mu } \left(t\right)\right]_{1}
\right)^{*} W\left(t,l_{\mu } ,m\right)F\left(t,l_{\mu }
,m\right)dt,\text{ if }f\left(t\right)\in H,\, \mathcal{I}\text{ is finite}}\text{ or }f(t)\in \overset{\circ}H
\end{array}\right.,
\end{equation}
$\mu \in \left[\alpha ,\beta \right]$.

Moreover, if vector-function $f\left(t\right)$ satisfy the
following conditions
\begin{gather}\label{star3}
\begin{matrix}
\mathop{{ P} }\limits^{\circ } E_{\infty } f=f,\, \, \mathop{{ P}
}\limits^{\circ } \int _{\mathbb{R}^{1}\setminus
\mathcal{B}^1}dE_{\mu } f =0\text{ if } f\in
\mathop{H}\limits^{{}^\circ },\ \mathcal{I}\text{ is infinite}
\\E_{\infty }f=f,\
\int_{\mathbb{R}^1\setminus \mathcal{B}^1}dE_{\mu } f =0\text{ if
}f\in H,\, \, \mathcal{I}\text{ is finite }
\end{matrix}
\end{gather}
then the inversion formulae in $L_{m}^{2} \left(\mathcal{I}\right)$
\begin{equation} \label{GEQ__87_}
\begin{matrix} {f\left(t\right)=\mathop{P}\limits^{\circ }
\int_{\mathcal{B}^1}\left[X_{\mu } \left(t\right)\right]_{1} d\sigma
\left(\mu \right)\varphi \left(\mu
,f\right)\text{ if }f\left(t\right)\in \mathop{H}\limits^{\circ } },\mathcal{I}\text{ is infinite}, \\
{f\left(t\right)=\int _{\mathcal{B}^1}^{}\left[X_{\mu }
\left(t\right)\right]_{1}  d\sigma \left(\mu \right)\varphi
\left(\mu ,f\right),\, \, if\, f\left(t\right)\in H,\mathcal{I}\,
\, is\, finite} \end{matrix}
\end{equation}
and Parceval's equality
\begin{equation} \label{GEQ__88_}
m\left[f,g\right]=\int _{\mathcal{B}^1}\left(d\sigma \left(\mu
\right)\varphi \left(\mu ,f\right),\varphi \left(\mu
,g\right)\right) ,
\end{equation}
are valid, where $g\left(t\right)\in \mathop{H}\limits^{\circ } $
if $\mathcal{I}$ is infinite or $g\left(t\right)\in H$, if
$\mathcal{I}$ is finite.

In general case for $f\left(t\right),\, g\left(t\right)\in
\mathop{H}\limits^{\circ }$ if $\mathcal{I}$ is infinite or
$f\left(t\right),g\left(t\right)\in {H}$ if $\mathcal{I}$ is
finite, the inequality of Bessel type
\begin{equation} \label{GEQ__89_}
m\left[f\left(t\right),g\left(t\right)\right]\le \int
_{\mathcal{B}^1}\left(d\sigma \left(\mu \right)\varphi \left(\mu
,f\right),\, \varphi \left(\mu ,g\right)\right)
\end{equation}
is valid.
\end{theorem}

Let us notice that $\mathcal{B}^1=\cup_k (a_k,b_k)$, $(a_j,b_j)\cap
(a_k,b_k)=\varnothing$, $k\not= j$ since $\mathcal{B}^1$ is an open
set. In \eqref{GEQ__87_}
$\mathop{P}\limits^{\circ}\int_{\mathcal{B}^1}=\sum_k\lim\limits_{\alpha_k\downarrow
a_k,\beta_k\uparrow
b_k}\mathop{P}\limits^{\circ}\int_{\alpha_k}^{\beta_k}$.  In
\eqref{GEQ__87_}-\eqref{GEQ__89_} we understand
$\int_{\mathcal{B}^1}$ similarly.

\begin{proof}
Let for definiteness $r=s=2n$, $\mathcal{I}$ is infinite (for
another cases the proof becomes simpler). Let the vector-functions
$f\left(t\right),\, g\left(t\right)\in \mathop{H}\limits^{\circ }
,\, \lambda =\mu +i\varepsilon ,G_{\lambda } \left(t,l_{\lambda }
,m\right)$ be defined by  \eqref{GEQ__23_} with
$g\left(t\right)$ instead of $f\left(t\right)$. In view of the
Stieltjes inversion formula, we have
\begin{multline} \label{GEQ__90_}
\left(E_{\alpha ,\beta } f,g\right)_{L_{m}^{2}
\left(\mathcal{I}\right)} =\mathop{\lim }\limits_{\varepsilon
\downarrow 0} \frac{1}{2\pi i} \int _{\alpha }^{\beta
}\left(\left[y_{1} \left(t,\lambda ,f\right)-y_{1}
\left(t,\bar{\lambda },f\right)\right],g\right)_{m} d\mu  =\\
=\mathop{\lim }\limits_{\varepsilon \downarrow 0} \frac{1}{2\pi i}
\int_{\alpha }^{\beta }\bigg[\left(\vec{y}_{1} \left(t,l_{\lambda
} ,m,f\right),G\left(t,l_{\lambda }
,m\right)\right)_{L_{W\left(t,l_{\lambda } ,m\right)}^{2}
\left(\mathcal{I}\right)} -\left(\vec{y}_{1}
\left(t,l_{\bar{\lambda }} ,m,f\right),G\left(t,l_{\bar{\lambda }}
,m\right)\right)_{L_{W\left(t,l_{\bar{\lambda }} ,m\right)}^{2}
\left(\mathcal{I}\right)} +\\
+2i\int _{\mathcal{I}}\left(\left(\Im  {p}_{n}^{ -1}
\left(t,\lambda \right)\right)f^{\left[n\right]}
\left(t\left|m\right. \right),g^{\left[n\right]}
\left(t\left|m\right. \right)\right)dt\bigg]d\mu  =\\
=\mathop{\lim }\limits_{\varepsilon \downarrow 0} \frac{1}{2\pi i}
\int _{\alpha }^{\beta }\left[\left( M\left(\lambda \right) \int
_{\mathcal{I}}X_{\bar{\lambda }}^{*}
\left(t\right)W\left(t,l_{\bar{\lambda }}
,m\right)F\left(t,l_{\bar{\lambda }} ,m\right)dt, \, \int
_{\mathcal{I}}X_{\lambda }^{*} \left(t\right)W\left(t,l_{\lambda }
,m\right)G\left(t,l_{\lambda } ,m\right)dt \right)\right.  -
\\
\left. -\left( M^{*} \left(\lambda \right) \int
_{\mathcal{I}}X_{\lambda }^{*} \left(t\right)W\left(t,l_{\lambda }
,m\right)F\left(t,l_{\lambda } ,m\right)dt ,\, \int
_{\mathcal{I}}X_{\bar{\lambda }}^{*}
\left(t\right)W\left(t,l_{\bar{\lambda} }
,m\right)G\left(t,l_{\bar{\lambda }} ,m\right)dt
\right)\right]d\mu=\\
=\int _{\alpha }^{\beta }\left(d\sigma \left(\mu \right)\int
_{\mathcal{I}}X_{\mu }^{*} \left(t\right)W\left(t,l_{\mu }
\right)F\left(t,l_{\mu } ,m\right)dt ,\, \int _{\mathcal{I}}X_{\mu
}^{*} \left(t\right)W\left(t,l_{\mu } \right)G_{\mu }
\left(t,l_{\mu } ,m\right)dt \right),
\end{multline}
where the second equality is a corollary of formula (40) from \cite{KhrabMAG,KhrabArxiv},
the next to last is a corollary of \eqref{GEQ__62_} and the
last one follows from the well-known generalization of the
Stieltjes inversion formula \cite[p.
803]{Shtraus1}, \cite[p. 952]{Bruk1}. (In the case of
finite $\mathcal{I}$ we have to substitute in \eqref{GEQ__90_}
$M(\lambda)$ by $P M(\lambda) P$ and then when passing to the next
to the last equality in \eqref{GEQ__90_} we have to use the
remark after the proof of Lemma 2.1 from \cite{KhrabMAG,KhrabArxiv}.) But for $\lambda \in
\mathcal{B}$
\begin{equation} \label{GEQ__91_}
\int_{\mathcal{I}}X_{\bar{\lambda }}^{*}
\left(t\right)W\left(t,l_{\bar{\lambda }}
,m\right)F\left(t,l_{\bar{\lambda }} \right)dt =\int
_{\mathcal{I}}\left(\left[X_{\bar\lambda } \left(t\right)\right]_{1}
\right)^{*} m\left[f\right]dt ,
\end{equation}
because in view of Theorem 2.1 from  \cite{KhrabMAG,KhrabArxiv}
\begin{multline*}
\forall h\in \mathcal{H}^{r} :(\int _{\mathcal{I}}X_{\bar{\lambda
}}^{*} \left(t\right)W\left(t,l_{\bar{\lambda }}
,m\right)F\left(t,l_{\bar{\lambda }} \right)dt,h )=\\
=\int _{\mathcal{I}}\left(W\left(t,l_{\bar{\lambda }}
,m\right)F\left(t,l_{\bar{\lambda }} \right),X_{\bar{\lambda }}
\left(t\right)h\right)dt =(\int
_{\mathcal{I}}\left(\left[X_{\bar{\lambda }} \right]_{1}
\right)^{*} m\left[f\right],h)dt.
\end{multline*}
Due to \eqref{GEQ__90_}, \eqref{GEQ__91_},
\eqref{GEQ__85_}
\begin{equation} \label{GEQ__92_}
\left(E_{\alpha ,\beta } f,g\right)_{L_{m}^{2}
\left(\mathcal{I}\right)} =\int _{\alpha }^{\beta }\left(d\sigma
\left(\mu \right)\varphi \left(\mu ,f\right),\varphi \left(\mu
,g\right)\right) .
\end{equation}

The equality \eqref{GEQ__88_} and inequality
\eqref{GEQ__89_} are the corollaries of \eqref{GEQ__92_}.

Representing $\varphi \left(\mu ,g\right)$ in \eqref{GEQ__92_}
by the second variant of \eqref{GEQ__85_}, changing in
\eqref{GEQ__92_} the order of integration and replacing
$\int_{\alpha}^\beta$ by integral sum and using
Theorem 2.1 from  \cite{KhrabMAG,KhrabArxiv} we obtain that
\begin{multline*}
\left(E_{\alpha ,\beta } f,g\right)_{L_{m}^{2}
\left(\mathcal{I}\right)} =(\int _{\alpha }^{\beta }\left[X_{\mu }
\left(t\right)\right]_{1} d\sigma \left(\mu \right)\varphi
\left(\mu ,f\right),g\left(t\right) )_{L_{m}^{2}
\left(\mathcal{I}\right)} =\\
=\int _{\alpha }^{\beta }\left[X_{\mu } \left(t\right)\right]_{1}
d\sigma \left(\mu \right)\varphi \left(\mu
,f\right),g\left(t\right) )_{L_{m}^{2} \left(\mathcal{I}\right)}
\end{multline*}
and \eqref{GEQ__84_} is proved since
$g(t)\in\overset{\circ}{H}$. Equalities \eqref{GEQ__87_} are
the corollary of \eqref{GEQ__84_}, \eqref{star3}. Theorem
\ref{th7} is proved.
\end{proof}

Let us notice that if
$L_m^2(\mathcal{I})=\overset{\circ\ }{L_m^2}(\mathcal{I})$ then
Theorem \ref{th7} is valid without condition \eqref{GEQ__55_}
with $P=I_r$ if $\mathcal{I}$ is infinite.

Formulae (\ref{GEQ__84_}), (\ref{GEQ__87_}), (\ref{GEQ__88_}) are similar to corresponding formulas for scalar differential operators from \cite{Shtraus1} (operator case see \cite{Bruk1}). For such operators the formulas that corresponds to (\ref{GEQ__84_}), (\ref{GEQ__87_}), (\ref{GEQ__88_})  are represented in \cite[p. 251, 255]{DS}, \cite[p. 516]{Naimark} in another form. Let us represent for example inversion formula (\ref{GEQ__87_}),  in the form analogues to \cite{DS,Naimark}.

\begin{proposition}\label{prop21}
Let all conditions of Theorem \ref{th7} hold. Let us represent spectral operator-function $\sigma \left(\mu \right)$ in matrix form: $\sigma \left(\mu \right)=\left\| \sigma _{ij} \left(\mu \right)\right\| _{i,j=1}^{r},\ \sigma _{ij} \left(\mu \right)\in B\left(\mathcal{H}\right)$. Then the following inversion formulae in $L_{m}^{2} \left(\mathcal{I}\right)$
\begin{gather}
\label{prop21_1}
f\left(t\right)=\mathop{P}\limits^{\circ } \int _{\mathcal{B}^1}\sum _{i,j=1}^{r}x_{i} \left(t,\mu \right)d\sigma _{ij} \left(\mu \right)  \int _{\mathcal{I}}x_{j}^{*} \left(s,\mu \right)m\left[f\left(s\right)\right] ds,\\\label{prop21_2}
f\left(t\right)=\mathop{P}\limits^{\circ } \int _{\mathcal{B}^1}\sum _{i,j=1}^{r}x_{i} \left(t,\mu \right) d\sigma _{ij} \left(\mu \right)  \int _{\mathcal{I}}\sum _{k=0}^{s/2}\left(x_{j}^{\left(k\right)} \left(s,\mu \right)\right)^{*}  \mathrm{m}_{k}  \left[f\left(s\right)\right]ds
\end{gather}
are valid if $\mathcal{I}$ is infinite, vector function $f\left(t\right)\in \mathop{H}\limits^{\circ } $ satisfies \eqref{star3}. Let $\mathcal{I}$ is finite. Then $\mathop{P}\limits^{\circ } $ in (\ref{prop21_1})-(\ref{prop21_2}) disappears, and (\ref{prop21_1}) (respectively, (\ref{prop21_2})) is valid for vector functions $f\left(t\right)\in \mathop{H}\limits^{\circ } $ (respectively, $f\left(t\right)\in H$) satisfying (\ref{star3}).
In formulae (\ref{prop21_1}), (\ref{prop21_2}): $x_i(t,\mu)\in B(\mathcal{H})$, 
$(x_1(t,\mu),\dots,x_r(t,\mu))=[X_\mu(t)]_1$, $\mathrm{m}_k[f(s)]$ see \eqref{mk}.
\end{proposition} 

The proof of this proposition is carried out in the same way as the proof of (\ref{GEQ__87_}) taking into account the proof of Remark 3.1 from \cite{KhrabMAG}.

Further we present several statements which allow to check the fulfilment of conditions (\ref{star3}) of Theorem \ref{th7} in various situations.

It is known (see for example \cite{Khrab2,Khrab3} or \cite[Ex. 3.2]{KhrabMAG,KhrabArxiv})
that even in the case $n_{\lambda } \left[y\right]\equiv 0$ in
\eqref{GEQ__1_}, \eqref{GEQ__2_} there is such $E_{\mu} $
satisfying \eqref{GEQ__3_}, \eqref{GEQ__62_}-\eqref{GEQ__66_}, \eqref{GEQ__83_} that ${E} _{\infty }
\ne \mathbf{I}$.

On the other hand if $n_\lambda[y]\equiv 0$ then $R(\lambda)$ is a generalized resolvent of relation $\mathcal{L}_0$ and  $\forall f\in
D\left(\mathcal{L}_{0} \right){E} _{\infty } f=f$ in view of
\cite{Khrab1,Khrab3}. Here $\mathcal{L}_{0}$ is the minimal relation generated in $L^2_m(\mathcal{I})$ by the pair of expressions $l[y]$ and $m[y]$; in particular $\mathcal{L}_0\supset \{\{y(t),f(t)\}:\ y(t)\in C_0^r(\mathcal{I}),f(t)\in \overset{\, \circ}H,l[y]=m[f]\}$  (see \cite{Khrab6,KhrabArxiv,KhrabMAG}).

Let expression $n_{\lambda } $ in representation
\eqref{GEQ__2_}, \eqref{GEQ__561_} have a divergent form
with coefficients $\tilde{\tilde{p}}_{j} =\tilde{\tilde{p}}_{j}
\left(t,\lambda \right),\tilde{\tilde{q}}_{j}
=\tilde{\tilde{q}}_{j} \left(t,\lambda
\right),\tilde{\tilde{s}}_{j} =\tilde{\tilde{s}}_{j}
\left(t,\lambda \right)$.

We denote $m\left(t\right)$ three-diagonal $\left(n+1\right)\times
\left(n+1\right)$ operator matrix, whose elements under main
diagonal are equal to $\left(-\frac{i}{2} \tilde{q}_{1} ,\, \ldots
,\, -\frac{i}{2} \tilde{q}_{n} \right)$, the elements over the
main diagonal are equal to $\left(\frac{i}{2} \tilde{s}_{1} ,\,
\ldots ,\, \frac{i}{2} \tilde{s}_{n} \right)$, the elements on the
main diagonal are equal to $\left(\tilde{p}_{0} ,\, \ldots ,\,
\tilde{p}_{n} \right)$, where $\tilde{p}_{j} ,\tilde{q}_{j}
,\tilde{s}_{j} =\tilde{q}_{j}^{*} $ are the coefficients of
expressions $m$. (Here either $2n$ or $2n+1$ is equal to the order
$r$ of $l_\lambda$). If order of $n_{\lambda } $ is less or equal
to $2n$, we denote $n\left(t,\lambda \right)$ the analogues
$\left(n+1\right)\times \left(n+1\right)$ operator matrix with
$\tilde{\tilde{p}}_{j} ,\tilde{\tilde{q}}_{j}
,\tilde{\tilde{s}}_{j} $ instead of $\tilde{p}_{j} ,\tilde{q}_{j}
,\tilde{s}_{j} $. If order $m$ or order $n_{\lambda } $ is less
than $2n$, we set the correspondent elements of $m\left(t\right)$
or $n\left(t,\lambda \right)$ be equal to zero.

\begin{theorem}\label{th8}
Let in \eqref{GEQ__1_}, \eqref{GEQ__2_} the order of the
expression $n_{\lambda }[y] $ is less or equal to the order of the
expression $(l-\lambda m)[y]$ (and therefore in view of
(\ref{GEQ__561_}) the order of $l-\lambda m$ is equal to $r$;
so $Q\left(t,l_{\lambda } \right)=Q\left(t,l-\lambda m\right)$).
Let $y=R_{\lambda } f,\, f\in H$ be the generalized resolvent of
the relation $\mathcal{L}_{0} $ and $y$ satisfy equation
\eqref{GEQ__1_}. Let $y_{1} =R\left(\lambda \right)f,\, f\in
H$ be the operator \eqref{GEQ__62_}, \eqref{GEQ__63_} from
Theorem \ref{th4}.

Let the following conditions hold for $\tau >0$ large enough:

1$^\circ$.

{\small
\begin{multline} \label{GEQ__94_}
\left.\mathop{\lim }\limits_{\alpha \downarrow a,\, \beta \uparrow
b} \frac{(\Re Q\left(t,l_{\lambda } \right)\left(\vec{y}_{1}
\left(t,l_{\lambda } ,m,f\right)-\vec{y}\left(t,l-\lambda
m,m,f\right)),\left(\vec{y}_{1} \left(t,l_{\lambda }
,m,f\right)\right)-\vec{y}\left(t,l-\lambda
m,m,f\right)\right)}{\Im \lambda }\right|_\alpha^\beta\le\\\le 0,\
\lambda =i\tau
\end{multline}
}

2$^\circ$.
\begin{equation} \label{GEQ__95_}
\Im{n}\left({\rm t,}\lambda \right)\le c\left(t,\tau
\right)m\left(t\right),\, t\in \bar{\mathcal{I}},\, \, \lambda
=i\tau ,
\end{equation}
where the scalar function $c\left(t,\tau \right)$ satisfles the
following condition:
\begin{equation} \label{GEQ__96_}
\mathop{\sup }\limits_{t\in \bar{\mathcal{I}}} c\left(t,\tau
\right)=o\left(\tau \right),\, \, \, \tau \to +\infty .
\end{equation}

Then for generalized spectral family ${E}_{\mu } $
\eqref{GEQ__83_} corresponding by \eqref{GEQ__3_} to the
resolvent $R\left(\lambda \right)$
\eqref{GEQ__62_}-\eqref{GEQ__63_} from Theorem \ref{th4}
and for generalized spectral family $\mathcal{E}_{\mu } $
corresponding to the generalized resolvent $R_{\lambda } $ one has
${E} _{\infty } =\mathcal{E} _{\infty } $.

\end{theorem}

Let us notice that in view of \eqref{GEQ__95_} the coefficient
at the highest derivative in the expression $l-\lambda m$ has
inverse from $B\left(\mathcal{H}\right)$ if $t\in
\bar{\mathcal{I}}$, $\Im \lambda \ne 0$.

\begin{proof}
Let $f\left(t\right)\in H$, $y_{1} =R\left(\lambda \right)f$,
$y=R_{\lambda } f$. Then $z=y_{1} -y$ satisfies the following
equation
\begin{equation} \label{GEQ__97_}
l\left[z\right]-\lambda m\left[z\right]=n_{\lambda } \left[y_{1} \right].
\end{equation}

Applying to the equation \eqref{GEQ__97_} the Green formula
\cite[Theorem 1.3]{KhrabMAG,KhrabArxiv} one has
$$
\left.\int_{\alpha }^{\beta }\Im \left(n_{\lambda } \left\{y_{1}
,z\right\}\right)dt +\int _{\alpha }^{\beta }m\left\{z,z\right\}dt
=\frac{1}{2} \frac{\Re \left(Q\left(t,l_{\lambda }
\right)\vec{z},\vec{z}\right)}{\Im \lambda } \right|_\alpha^\beta
,
$$
where $\vec{z}=\vec{z}\left(t,l-\lambda m,n_{\lambda } ,y_{1}
\right)=\vec{y}_1\left(t,l_{\lambda }
,m,f\right)-\vec{y}\left(t,l-\lambda m,m,f\right)$ in view Lemma 1.2 from \cite{KhrabMAG,KhrabArxiv} and of
\eqref{GEQ__25_}. Hence for $\tau >0$
large enough
\begin{multline} \label{GEQ__98_}
m\left[z,z\right]\le -\int_{\mathcal{I}}\Im\left(n_{\lambda }
\left[y_{1} ,z\right]\right)dt/\tau
 \leq\\
 \leq{\int _{\mathcal{I}}\left|\left({ n}\left({\rm t,}\lambda
\right)col\left\{y_{1} ,y'_{1} ,\ldots ,y_{1}^{\left(n\right)}
\right\},col\left\{z,z',\ldots ,z^{\left(n\right)}
\right\}\right)\right| dt \mathord{\left/{\vphantom{\int
_{\mathcal{I}}\left|\left({\rm n}\left({\rm t,}\lambda
\right)col\left\{y_{1} ,y'_{1} ,\ldots ,y_{1}^{\left(n\right)}
\right\},col\left\{z,z',\ldots ,z^{\left(n\right)}
\right\}\right)\right| dt \tau }}\right.\kern-\nulldelimiterspace}
\tau } ,\quad \lambda =i\tau
\end{multline}
in view of \eqref{GEQ__94_}. But due to the inequality of the
Cauchy type for dissipative operators \cite[p. 199]{Nagy} and
\eqref{GEQ__95_}, \eqref{GEQ__96_}: subintegral function
in the last integral in \eqref{GEQ__98_} is less or equal to
$\left(m\left\{z,z\right\}\right)^{1/2}
\left(m\left\{y_1,y_1\right\}\right)^{1/2} o\left(1\right)$ with
$\lambda =i\tau ,\, \tau \to +\infty $. Therefore $\left\|
z\right\|_m \le o\left({1/\tau}\right)\left\| f\right\| _m $ since
$\left\| R_{\lambda } \right\| \le {1 \mathord{\left/{\vphantom{1
\left|\Im \lambda \right|}}\right.\kern-\nulldelimiterspace}
\left|\Im \lambda \right|} $. Hence
$$\left\| R\left(\lambda
\right)-R_{\lambda } \right\|\le o\left({1
\mathord{\left/{\vphantom{1 \tau
}}\right.\kern-\nulldelimiterspace} \tau } \right),\quad \lambda
=i\tau ,\quad \tau \to +\infty$$ To complete the proof of the
theorem it remains to prove the following
\begin{lemma}\label{lm10+}
Let $R_{k} \left(\lambda \right)=\int
_{\mathbb{R}^1}\frac{dE^k_{\mu } }{\mu -\lambda }  ,\, k=1,2$,
where $E_{\mu }^{k} $ are the generalized spectral families the
type \eqref{GEQ__83_} in Hilbert space $\mathbf{H}$. If
$\left\| R_{1} \left(\lambda \right)-R_{2} \left(\lambda
\right)\right\| \le o\left({1 \mathord{\left/{\vphantom{1 \tau
}}\right.\kern-\nulldelimiterspace} \tau } \right),\, \, \,
\lambda =i\tau ,\, \, \tau \to +\infty $, then $E_{\infty }^{1}
=E_{\infty }^{2} $.
\end{lemma}

\begin{proof}
Let $f\in $\textbf{H} $\, \sigma \left(\mu
\right)=\left(\left(E_{\mu }^{1} -E_{\mu }^{2} \right)f,f\right)$.
One has
\begin{multline*}
\left|\left(\left[R_{1} \left(\lambda \right)-R_{2} \left(\lambda
\right)\right]f,f\right)\right|=\\=\frac{1}{\tau } \left|-\int
_{\Delta }d\sigma \left(\mu \right) +\int _{\Delta }\frac{\mu
d\sigma \left(\mu \right)}{\mu -\lambda }  +\int _{R^{1}
\backslash \Delta }\frac{\lambda d\sigma \left(\mu \right)}{\mu
-\lambda }  \right|\le o\left({1 \mathord{\left/{\vphantom{1 \tau
}}\right.\kern-\nulldelimiterspace} \tau } \right)\left\|
f\right\| ^{2} ,\, \, \lambda =i\tau ,\, \, \tau \to +\infty
\end{multline*}
Therefore
\begin{equation} \label{GEQ__99_}
\left|-\int _{\Delta }d\sigma \left(\mu \right)+\int _{\Delta
}\frac{\mu d\sigma \left(\mu \right)}{\mu -\lambda }  +\int
_{R^{1} \backslash \Delta }\frac{\lambda d\sigma \left(\mu
\right)}{\mu -\lambda }   \right|\le o\left(1\right),\, \, \,
\lambda =i\tau ,\, \, \tau \to +\infty .
\end{equation}
For an arbitrarily small $\varepsilon >0$ we choose such finite
interval $\Delta \left(\varepsilon \right)$ that for any finite
interval $\Delta \supseteq \Delta \left(\varepsilon
\right):\left|\int _{R^{1} \backslash \Delta }\frac{\lambda
d\sigma \left(\mu \right)}{\mu -\lambda }
\right|<\frac{\varepsilon }{2} ,\, \, \lambda =i\tau $. But for
any finite interval $\Delta \supseteq \Delta \left(\varepsilon
\right)\, \exists N=N\left(\Delta \right):\forall \tau
>N:\left|\int _{\Delta }\frac{\mu d\sigma \left(\mu \right)}{\mu
-\lambda }  \right|<\frac{\varepsilon }{2} ,\, \, \lambda =i\tau
$. Therefore $\forall \varepsilon >0,\Delta \supseteq \Delta
\left(\varepsilon \right):\left|\int _{\Delta }d\sigma \left(\mu
\right) \right|<\varepsilon $ in view of \eqref{GEQ__99_}.
Hence $\forall f\in $\textbf{H}$:\left(E_{\infty }^{1}
f,f\right)=\left(E_{\infty }^{2} f,f\right)\Rightarrow E_{\infty
}^{1} =E_{\infty }^{2} $.

Lemma \ref{lm10+} and Theorem \ref{th8} are proved.
\end{proof}
\end{proof}

\begin{corollary}\label{cor3}
Let the conditions of Theorems \ref{th7}, \ref{th8} hold. Then for
generalized spectral family $E_{\mu } $ from Theorem \ref{th7}
$\forall f\left(t\right)\in D\left(\mathcal{L}_{0}
\right):E_{\infty } f=f$.
\end{corollary}

\begin{remark}
\label{rem3} If $L_{m}^{2}
\left(\mathcal{I}\right)=\mathop{L_{m}^{2} }\limits^{\circ }
\left(\mathcal{I}\right)$, then it is sufficient to verify
condition \eqref{GEQ__94_} in Theorem \ref{th8} for $f\in
\mathop{H}\limits^{\circ } $.
\end{remark}

\begin{proposition}
\label{prop2} Let the order of expression $n_{\lambda } $ be less
or equal to the order of expression $l-\lambda m$ and the
coefficient of $l-\lambda m$ at the highest derivative has the
inverse from $B\left(\mathcal{H}\right)$ for $t\in
\bar{\mathcal{I}},\, \lambda \in \mathcal{B} \left(l-\lambda
{m}\right)$, where $\mathcal{B}\left(l-\lambda m\right)$ is an
analogue of the set $\mathcal{B} =\mathcal{B}\left(l_\lambda
\right)$. Let interval $\mathcal{I}$ be finite and for equation
\eqref{GEQ__1_}, \eqref{GEQ__2_} with $n_{\lambda }
\left[y\right]\equiv 0$ condition \eqref{GEQ__55_} holds with
${P} ={I} _{r}$. 

Then for equation \eqref{GEQ__1_}-\eqref{GEQ__2_} this condition also holds with $P=I_r$. 

The boundary value problem which is obtained by adding to equation \eqref{GEQ__1_}-\eqref{GEQ__2_} with $n_\lambda[y]\equiv 0$ (respectively, \eqref{GEQ__1_}-\eqref{GEQ__2_}) boundary conditions 
\begin{gather}\label{bk1}
\exists h=h\left(\lambda ,f\right)\in \mathcal{H}^{r} :\, \,
\vec{y}\left(a,l_\lambda-\lambda m ,m,f\right)=\mathcal{M}_{\lambda } h,\, \,
\, \vec{y}\left(b,l_\lambda-\lambda m ,m,f\right)=\mathcal{N}_{\lambda }
h\\
\label{bk2}
\text{(respectively, }\exists h_1=h_1\left(\lambda ,f\right)\in \mathcal{H}^{r} :\, \,
\vec{y}\left(a,l_\lambda ,m,f\right)=\mathcal{M}_{\lambda } h_1,\, \,
\, \vec{y}\left(b,l_\lambda ,m,f\right)=\mathcal{N}_{\lambda }
h_1\text{)},\end{gather}
has the unique solutions $y=R_{\lambda } f$ (respectively, $y_{1} =R\left(\lambda \right)f$) for an arbitrary $f(t)\in H$.
Here the operator-functions $\mathcal{M}_{\lambda },
\mathcal{N}_{\lambda} \in B\left(\mathcal{H}^{r} \right)$ depend
analytically on the non-real $\lambda $,
\begin{equation*}
\mathcal{M}_{\bar{\lambda }}^{*} \left[\Re Q\left(a,l_{\lambda }
\right)\right]\mathcal{M}_{\lambda } =\mathcal{N}_{\bar{\lambda
}}^{*} \left[\Re Q\left(b,l_{\lambda }
\right)\right]\mathcal{N}_{\lambda },\ \Im  \lambda \ne
0,
\end{equation*}
where $Q\left(t,l_{\lambda } \right)$ is the coefficient of
equation \eqref{GEQ__54_} corresponding by Theorem \ref{th1}
to equation \eqref{GEQ__1_},
\begin{equation*}
\left\| \mathcal{M}_{\lambda } h\right\| +\left\|
\mathcal{N}_{\lambda } h\right\| >0,\ 0\ne h\in
\mathcal{H}^{r} ,\, \Im \lambda \ne 0,
\end{equation*}
the lineal $\left\{\mathcal{M}_{\lambda } h\oplus
\mathcal{N}_{\lambda } h\left|h\in \mathcal{H}^{r} \right.
\right\}\subset \mathcal{H}^{2r} $ is a maximal
$\mathcal{Q}$-nonnegative subspace if $\Im  \lambda \ne 0$, where
$\mathcal{Q}=\left(\Im \lambda \right)\mathrm{diag} \left(\Re
Q\left(a,l_{\lambda } \right),\, -\Re Q\left(b,l_{\lambda }
\right)\right)$ (and therefore
\begin{equation*}
\Im  \lambda \left(\mathcal{N}_{\lambda }^{*} \left[\Re
Q\left(b,l_{\lambda } \right)\right]\mathcal{N}_{\lambda }
-\mathcal{M}_{\lambda }^{*} \left[\Re Q\left(a,l_{\lambda }
\right)\right]\mathcal{M}_{\lambda } \right)\le 0,\ \left.
\Im  \lambda \ne 0\right).
\end{equation*}

Operator $R_\lambda f$ (respectively, $R(\lambda)f$) is a generalized resolvent of 
$\mathcal{L}_0$ (respectively, a resolvent of \eqref{GEQ__63_} type). This resolvent is constructed by applying  of Theorem \ref{th4} with the characteristic operator
\begin{gather}\label{mlambda+}
M(\lambda)=-{1\over 2}\left(X_\lambda^{-1}(a)\mathcal{M}_\lambda+X_\lambda^{-1}(b)\mathcal{N}_\lambda\right)\left(X_\lambda^{-1}(a)\mathcal{M}_\lambda-X_\lambda^{-1}(b)\mathcal{N}_\lambda\right)^{-1}(iG)^{-1},
\end{gather}
to equation \eqref{GEQ__1_}, \eqref{GEQ__2_} with $n_\lambda[y]\equiv 0$ (respectively, \eqref{GEQ__1_}-\eqref{GEQ__2_}). Here $\left(X_\lambda^{-1}(a)\mathcal{M}_\lambda-X_\lambda^{-1}(b)\mathcal{N}_\lambda\right)^{-1}\in B(\mathcal{H}^r)$, $\Im\lambda\not= 0$, $X_\lambda(t)$ is an operator solution from Theorem \ref{th4} which corresponds to equation $(l-\lambda m)[y]=0$ (respectively, $l_\lambda[y]=0$). 
 
The resolvents $y=R_{\lambda } f$ and $y_{1} =R\left(\lambda \right)f$ satisfy condition \eqref{GEQ__94_} of Theorem \ref{th8}.

\end{proposition}

Let us notice that if $\mathcal{I}$ is finite and condition \eqref{GEQ__55_} holds with $P=I_r$ then $M(\lambda)$ \eqref{mlambda+} is a characteristic operator of equation \eqref{GEQ__5_} and any characteristic operator of equation \eqref{GEQ__5_} has representation (\ref{mlambda+}). Also we notice that if $\mathcal{I}$ is finite, $n_\lambda[y]\equiv 0$ and condition \eqref{GEQ__55_} holds with $P=I_r$ then any generalized resolvent of $\mathcal{L}_0$ can be constructed as an operator $R(\lambda)$ from Theorem \ref{th4} in view of \cite{KhrabMAG,KhrabArxiv}.

\begin{proof}

In view of Theorems 3.2, 3.3 from \cite{KhrabMAG,KhrabArxiv}  it is sufficient to prove only
proposition about condition \eqref{GEQ__55_}.

Let for definiteness order $l=$ order $m=$ order $n_{\lambda }
=2n$.

Let for equation \eqref{GEQ__1_}-\eqref{GEQ__2_} with
$n_{\lambda } \left[y\right]\equiv 0$ condition
\eqref{GEQ__55_} with ${P} ={ I}_{r}$ hold, but for equation
\eqref{GEQ__1_}, \eqref{GEQ__2_} that is not true. Then in
view of \cite{Khrab5} the solutions $y_k\left(t\right)$ of
equation \eqref{GEQ__1_}-\eqref{GEQ__2_} with $f(t)=0$, $\lambda
=i$ exist for which
\begin{equation} \label{GEQ__100_}
\int _{\alpha }^{\beta }\left(m+\Im n_{i} \right)\left\{y_{k}
,y_{k} \right\} dt\to 0,\quad \vec{y}_{k} \left(0,l_{i}
,m,0\right)=f_{k} ,\; \left\| f_{k} \right\| =1,
\end{equation}
where $i\Im n_i =n_{i} $ in view of \eqref{GEQ__561_}. Hence
in view of  Theorem 1.2 from \cite{KhrabMAG,KhrabArxiv}) 
\begin{equation} \label{GEQ__101_}
\int _{\alpha }^{\beta }\left(W_{i} \left(t,l+im,n_{i} \right)
Y_{k} \left(t,l+im,n_i\right),Y_{k}\left(t,l+im,n_i\right)
\right)dt = \int _{\alpha }^{\beta }n_{i} \left\{y_{k} ,y_{k}
\right\}dt \to 0.
\end{equation}

On the other hand
\begin{equation} \label{GEQ__102_}
X_{i} \left(t\right)f_{k} =\tilde X_{i} \left(t\right)f_{k}
+\tilde X_{i} \left(t\right)\int _{o}^{t}\tilde X_{i}^{-1}
\left(s\right)J^{-1} W\left(s,l+im,n_{i} \right)Y_{k}
\left(s,l+im,n_i\right) ds .
\end{equation}
in view of Theorem \ref{th1} and the fact that
$\vec{y}_k(t,l-im,n_i,y_k)=\vec{y}_k(t,l_i,m,0)$, where $\tilde
X_{\lambda } \left(t\right)$ is an analogue of $X_{\lambda }
\left(t\right)$ for the case $n_{\lambda } \left[y\right]\equiv
0$.

Comparing \eqref{GEQ__101_}, \eqref{GEQ__102_} we see that
\begin{equation} \label{GEQ__103_}
\left\|X_{i} \left(t\right)f_{k} -\tilde X_{i} \left(t\right)f_{k}
\right\| \to 0
\end{equation}
uniformly in $t\in \left[\alpha ,\beta \right]$.

In view of \eqref{GEQ__100_} subsequence $y_{k_{q} } $ exist
such that
\begin{equation} \label{GEQ__104_}
m\left\{y_{k_{q} } ,y_{k_{q} } \right\}\mathop{\to }\limits^{a.a.} 0,\quad n_{i} \left\{y_{k_{q} } ,y_{k_{q} } \right\}\mathop{\to }\limits^{a.a.} 0.
\end{equation}
Due to second proposition in \eqref{GEQ__104_} and the arguments
as in the proof of Proposition 3.1 from \cite{KhrabMAG}  one has
\begin{equation} \label{GEQ__105_}
y_{k_{q} }^{\left[j\right]} \left(t\left|n_{i} \right.
\right)\mathop{\to }\limits^{a.a.} 0\quad j=n,\ldots ,2n.
\end{equation}

Let us denote $\tilde y_{k_{q} } \left(t\right)=\tilde X_{i}
\left(t\right)f_{k_{q} } $. In view of Theorem \ref{th1} and
\eqref{GEQ__103_}
\begin{equation} \label{GEQ__106_}
\left\| y_{k_{q} }^{\left(j\right)} \left(t\right)-\tilde y_{k_{q}
}^{\left(j\right)} \left(t\right)\right\| \to 0,\quad j=1,\ldots
,n-1,
\end{equation}
\begin{multline} \label{GEQ__107_}
\left\| \left(p_{n} \left(t\right)-i\tilde{p}_{n}
\left(t\right)\right)\left[y_{k_{q} }^{\left(n\right)}
\left(t\right)-\tilde y_{k_{q} }^{\left(n\right)}
\left(t\right)\right]-\right.\\\left.-{i\over 2}(q_n(t)-i\tilde
q_n(t))\left[y_{k_q}^{(n-1)}(t)-\tilde y_{k_q}^{(n-1)}(t)\right]
-y_{k_{q} }^{\left[n\right]} \left(t\left|n_{i} \right.
\right)\right\| =\\=\left\| y_{k_{q} }^{\left[n\right]}
\left(t|l_i\right)-\tilde y_{k_{q} }^{\left[n\right]}
\left(t|l-im\right)\right\| \to 0
\end{multline}
uniformly in $t\in \left[\alpha ,\beta \right]$. Comparing
\eqref{GEQ__104_}, \eqref{GEQ__105_},
\eqref{GEQ__107_} and using $\left(p_{n}
\left(t\right)-i\tilde{p}_{n} \left(t\right)\right)^{-1} \in
B\left(\mathcal{H}\right)$ we have
\begin{equation} \label{GEQ__108_}
\left(\tilde{p}_{n} \left(t\right)\tilde y_{k_{q}
}^{\left(n\right)} \left(t\right),\tilde y_{k_{q}
}^{\left(n\right)} \left(t\right)\right)\mathop{\to
}\limits^{a.a.} 0.
\end{equation}
In view of \eqref{GEQ__106_}, \eqref{GEQ__108_},
\eqref{GEQ__104_}
$$
m\left\{\tilde y_{k_{q} } ,\tilde y_{k_{q} } \right\}\mathop{\to
}\limits^{a.a.} 0,\quad \vec{y}_{k_q}(0,l-im, m,0)=f_{k_q},
$$
that contradicts to the condition \eqref{GEQ__55_} with
$P=I_{r} $ for equation \eqref{GEQ__1_}, \eqref{GEQ__2_}
with $n_{\lambda } \left[y\right]\equiv 0$. Proposition
\ref{prop2} is proved.
\end{proof}


If the set ${\mathbb R}^{1} \backslash \mathcal{B}^1$ has no finite limit points then to verify the condition $\mathop{P}\limits^{\circ } \int \ldots =0 $ or $\int \ldots =0$ in \eqref{star3} we can use  the following proposition which is a corollary of Lemma from \cite[p. 789]{Shtraus2}.

\begin{proposition}\label{prop_str1}
Let $R\left(\lambda \right)=\int _{{\mathbb R}^{1} }\frac{dE_{\mu } }{\mu -\lambda }  $, where $E_{\mu } $ is generalized spectral family in Hilbert space $\mathbf{H} ;\, g\in \mathbf{H} $. If $\sigma $ is not a point of continuity of $E_{\mu} g$, then $\exists c\left(\sigma,g\right)>0$: $\left\| R\left(\sigma +i\tau \right)g\right\| \sim \frac{c\left(\sigma ,g\right)}{\left|\tau \right|} ,\, \, \, \tau \to 0$.
\end{proposition}

\begin{proof} Let $\Delta $ be a jump of $E_{\mu } $ in the point $\sigma $. Then $\Delta g\ne 0$,
\begin{gather}\label{prop_str1_1}
R\left(\sigma +i\tau \right)g=i\frac{\Delta }{\tau } g+\int _{{\mathbb R}^{1} }\frac{d\tilde{E}_{\mu } g}{\mu -\lambda } \text{ where }\tilde{E}_{\mu } =\begin{cases} E_{\mu } ,&\mu \leq\sigma \\ E_{\mu } -\Delta ,&\mu >\sigma  \end{cases}
\end{gather}
Since the second  term in (\ref{prop_str1_1}) is $o\left({\raise0.7ex\hbox{$ 1 $}\!\mathord{\left/ {\vphantom {1 \left|\tau \right|}} \right. \kern-\nulldelimiterspace}\!\lower0.7ex\hbox{$ \left|\tau \right| $}} \right)$ in view of \cite[p. 789]{Shtraus2}\footnote{ Lemma from \cite[789]{Shtraus2} is proved for families  $E_{\mu } $  with  $E_{\infty } =$identity operator. But analysis of its proof shows that it is valid in general case.} proposition in proved.

\end{proof}


In the next theorem $\mathcal{I}=\mathbb{R}^1$ and condition
\eqref{GEQ__55_} hold with $P=I_{r} $ both on the negative
semi-axis $\mathbb{R}^1_{-} $ (i.e. as $\mathcal{I}=\mathbb{R}^1_{-} $) and on the
positive semi-axis $\mathbb{R}^1_{+} $ (i.e. as $\mathcal{I}=\mathbb{R}^1_{+} $).

\begin{theorem}\label{th9}
Let $\mathcal{I}=\mathbb{R}^1$, the coefficient of the expression
$l_\lambda[y]$ \eqref{GEQ__2_} be periodic on each of the
semi-axes $\mathbb{R}^1_{-} $ and $\mathbb{R}^1_{+} $ with periods $T_{-} >0$ and $T_{+}
>0$ correspondingly. Then the spectrums of the monodromy operators
$X_{\lambda } \left(\pm T_{\pm } \right)$ ($X_{\lambda }
\left(t\right)$ is from Theorem \ref{th4}) do not intersect the
unit circle as $\Im  \lambda \ne 0$, the characteristic operator $M\left(\lambda
\right)$ of the equation \eqref{GEQ__5_} is unique and equal
to
\begin{equation} \label{GEQ__109_}
M\left(\lambda \right)=\left(\mathcal{P}\left(\lambda
\right)-\frac{1}{2} I_{r} \right)\, \left(iG\right)^{-1} \quad
\left(\Im  \lambda \ne 0\right),
\end{equation}
where the projection $\mathcal{P}\left(\lambda \right)=P_{+}
\left(\lambda \right)\left(P_{+} \left(\lambda \right)+P_{-}
\left(\lambda \right)\right)^{-1} ,\, P_{\pm } \left(\lambda
\right)$ are Riesz projections of the monodromy operators
$X_{\lambda } \left(\pm T_{\pm } \right)$ that correspond to their
spectrums lying inside the unit circle, $\left(P_{+} \left(\lambda
\right)+P_{-} \left(\lambda \right)\right)^{-1} \in
B\left(\mathcal{H}^{r} \right)$ as $\Im  \lambda \ne 0$.

Also let $\dim \mathcal{H}<\infty $, a finite interval $\Delta
\subseteq \mathcal{B}^1 $. Then in Theorem \ref{th7} $d\sigma
\left(\mu \right)=d\sigma _{ac} \left(\mu \right)+d\sigma _{d}
\left(\mu \right),\mu \in \Delta $. Here $\sigma _{ac} \left(\mu
\right)\in AC\left(\Delta \right)$ and, for $\mu \in \Delta $,
$$\sigma '_{ac} \left(\mu \right)=\frac{1}{2\pi } G^{-1}
\left(Q_{-}^{*} \left(\mu \right)GQ_{-} \left(\mu
\right)-Q_{+}^{*} \left(\mu \right)GQ_{+} \left(\mu
\right)\right)G^{-1}, $$ where the projections $Q_{\pm } \left(\mu
\right)=q_{\pm } \left(\mu \right)\left(P_{+} \left(\mu
\right)+P_{-} \left(\mu \right)\right)^{-1} ,\, q_{\pm } \left(\mu
\right)$ are Riesz projections of the monodromy matrixes $X_{\mu }
\left(\pm T_{\pm } \right)$ corresponding to the multiplicators
belonging to the unit circle and such that they are shifted inside
the unit circle as $\mu $ is shifted to the upper half plane,
$P_{\pm } \left(\mu \right)=P_{\pm } \left(\mu +i0\right);\,
\sigma _{d} \left(\mu \right)$ is a step function.
\end{theorem}

Let us notice that the sets on which $q_{\pm } \left(\mu
\right),P_{\pm } \left(\mu \right),\left(P_{+} \left(\mu
\right)+P_{-} \left(\mu \right)\right)^{-1} $ are not infinitely
differentiable do not have finite limit points $\in \mathcal{B}^1 $
as well as the set of points of increase of $\sigma _{d} \left(\mu
\right)$.

\begin{proof}
The proof of Theorem \ref{th9} is similar to that on in the case
$n_{\lambda } \left[y\right]\equiv 0$ \cite{Khrab6}.
\end{proof}

The following examples demonstrate effects that are the results of
appearance in $l_{\lambda } $ \eqref{GEQ__2_} of perturbation
$n_{\lambda } $ depending nonlinearly on $\lambda $.

In Examples \ref{ex3}, \ref{ex4} nonlinear in $\lambda $
perturbation does not change the type of the spectrum. In this
examples $\dim \mathcal{H}=1,\, \, m\left[y\right]=-y''+y$.
$L_{m}^{2} \left(\mathcal{I}\right)=\mathop{L_{m}^{2}
}\limits^{\circ } \left(\mathcal{I}\right)=W_{2}^{1,2}
\left(\mathbb{R}^{1} \right)$. In Example \ref{ex5} such
perturbation implies an appearance of spectral gap with
"eigenvalue" in this gap.

\begin{example}\label{ex3}
Let
$$
l_{\lambda } \left[y\right]=iy'-\lambda
\left(-y''+y\right)-\left(-\frac{h}{\lambda } y\right),\, \, \,
h\ge 0.$$ Here $\mathcal{B} =\mathbb{C}\setminus 0,\, \, E_{0}
=E_{+0} $, spectral matrix $\sigma \left(\mu \right)\in AC_{loc}
$,
\\
$\sigma '\left(\mu \right)=\frac{1}{2\pi } \left(\begin{array}{cc}
{\frac{2}{\sqrt{4h+1-4\mu ^{2} } } } & {0} \\ {0} & {\frac{1}{2}
\sqrt{4h+1-4\mu ^{2} } } \end{array}\right)$, as $\left|\mu
\right|<\sqrt{h+{1 \mathord{\left/{\vphantom{1
4}}\right.\kern-\nulldelimiterspace} 4} } $,
\\
$\sigma '\left(\mu \right)=0$, as $\left|\mu \right|>\sqrt{h+{1
\mathord{\left/{\vphantom{1 4}}\right.\kern-\nulldelimiterspace}
4} } $.

In Example \ref{ex3} nonlinear in $\lambda $ perturbation change
edges of spectral band.
\end{example}

\begin{example}\label{ex4}
Let
$$
l_{\lambda } \left[y\right]=y^{\left(IV\right)} -\lambda
\left(-y''+y\right)-\left(-\frac{h}{\lambda } y\right),\, \, \,
h\ge 0.$$ Here $\mathcal{B} =\begin{cases}
\mathbb{C}\setminus\{0\},&h\ne 0 \\ \mathbb{C},&h=0
\end{cases},\, \, \, E_{0} =E_{+0} $, spectral matrix
$\sigma \left(\mu \right)\in AC_{loc} $,
\begin{gather*}
\sigma '\left(\mu \right)=\begin{cases} \frac{1}{2\pi }
\sqrt{\frac{\lambda +\sqrt{D} }{D} } \left(\begin{array}{cccc}
{\frac{2}{\lambda +\sqrt{D} } } & {0} & {0} & {-1} \\ {0} & {1} &
{\frac{-\lambda +\sqrt{D} }{2} } & {0} \\ {0} & {\frac{-\lambda
+\sqrt{D} }{2} } & {\left(\frac{\lambda -\sqrt{D} }{2} \right)^{2}
} & {0} \\ {-1} & {0} & {0} & {\frac{\lambda +\sqrt{D} }{2} }
\end{array}\right),&\text{ as }-\sqrt{h} <\mu <0,\, \mu
>\sqrt{h}  \\ \frac{1}{2\pi } \cdot \frac{1}{\sqrt{\lambda
-2\sqrt{q} } } \left(\begin{array}{cccc} {\frac{1}{\sqrt{q} } } &
{0} & {0} & {-1} \\ {0} & {1} & {-\sqrt{q} } & {0} \\ {0} &
{-\sqrt{q} } & {\sqrt{q} \left(\lambda -\sqrt{q} \right)} & {0} \\
{-1} & {0} & {0} & {\lambda -\sqrt{q} }
\end{array}\right),&\text{ as }\mu ^{*} <\mu <\sqrt{h},
\end{cases},
\end{gather*}
where $D=\mu ^{2} -4q,\, q={h/\mu} -\mu ,\,
\mu ^{*} =\mu ^{*} \left(h\right)$ - nonnegative root of equation
$D=0$. $\sigma '\left(\mu \right)=0$, as $\mu \notin
\left[-\sqrt{h} ,0\right]\bigcup \left[\mu ^{*} ,\infty \right)$.

In Example \ref{ex4} nonlinear in $\lambda $ perturbation implies
an appearance of additional spectral band $\left[-h,0\right]$,
variation of edge of semi-infinite spectral band and appearance
of interval $\left(\mu ^{*} ,\sqrt{h} \right)$ of fourfold
spectrum.
\end{example}

\begin{example}\label{ex5}
Let $\dim \mathcal{H}=2$,
$$l_{\lambda } \left[y\right]=\left(\begin{array}{cc} {0} & {-1} \\
{1} & {0}
\end{array}\right)y'-\lambda y-\left(\begin{array}{cc} {-{h
\mathord{\left/{\vphantom{h \lambda
}}\right.\kern-\nulldelimiterspace} \lambda } } & {0} \\ {0} & {0}
\end{array}\right)y,\quad h\ge 0.$$
Here $\mathcal{B} =\begin{cases} \mathbb{C}\setminus\{0\},& h\ne 0 \\
\mathbb{C},& h=0 \end{cases},$ spectral matrix $\sigma \left(\mu
\right)=\sigma _{ac} \left(\mu \right)+\sigma _{d} \left(\mu
\right)$, $\sigma_{ac}(\mu)\in AC_{loc},  \sigma '_{ac} \left(\mu
\right)\ne 0$, as $\left|\mu \right|>\sqrt{h} ,\, \, \sigma '_{ac}
\left(\mu \right)=0$, as $\left|\mu \right|<\sqrt{h} $,
step-function $\sigma _{d} \left(\mu \right)$ has only one jump
$\left(\begin{array}{cc} {0} & {0} \\ {0} & {{\sqrt{h}
\mathord{\left/{\vphantom{\sqrt{h}
2}}\right.\kern-\nulldelimiterspace} 2} } \end{array}\right)$ in
point $\mu =0$ (inside of spectral gap). In this point
$$\left(E_{+0} -E_{0} \right)f=\left(\begin{array}{cc} {0} & {0} \\
{0} & {{\sqrt{h}  \mathord{\left/{\vphantom{\sqrt{h}
2}}\right.\kern-\nulldelimiterspace} 2} } \end{array}\right)\int
_{-\infty }^{\infty }e^{-\sqrt{h} \left|t-s\right|}
f\left(s\right) ds,\, \, \, f\left(t\right)\in L^{2}
\left(\mathbb{R}^{1} \right).$$
\end{example}

Let us explain that in Examples \ref{ex3}, \ref{ex4}: 1) Spectral matices are locally absolutely continions in view of Theorem 3.6 and estimates of the type $\left\| M\left(\lambda \right)\right\| \sim \frac{c}{\left|\lambda \right|^{\alpha } } \, \left(\lambda \to i0\right),\, \alpha <1$ for corresponding characteristic operators  (cf. \cite{Khrab1}) that follows from \eqref{GEQ__109_}; 2) Equalities $E_{0} =E_{+0} $ follow from Proposition \ref{prop_str1}, equality $L_{m}^{2} \left(\mathbb{R}^{1} \right)=\mathop{L_{m}^{2} }\limits^{\circ } \left(\mathbb{R}^{1} \right)$
 and estimates of the type $\left\| R\left(i\tau \right)g\right\| _{m} \le \frac{c\left(g\right)}{\left|\tau \right|^{\beta } } ,\, \left(\tau \to 0\right),\, \beta <1,\, g\in \mathop{{\rm H} }\limits^{\circ } $ that follows from Theorems \ref{th4}, \ref{th9} and Floquet Theorem. 

Let us notice that in view of Floquet theorem conditions of
Theorem \ref{th8} ((\ref{GEQ__94_}) with account of Remark
\ref{rem3}) hold for all Examples \ref{ex3}-\ref{ex5}.

The following theorem is a generalization of results from \cite{Shtraus2} on the expansion in solutions of scalar Sturm-Louville equation which satisfy in regular end point the boudary condition depending on spectral parameter.

\begin{theorem} \label{th24}
Let $r=2n,\ \mathcal{I}=\left(0,\infty \right)$, condition (\ref{GEQ__55_})  with $P=\mathcal{I}_{2n} $ hold. Let contraction $v\left(\lambda \right)\in B\left(\mathcal{H}^{n} \right)$ satisfy the conditions of Lemma \ref{lm12}. Let $v\left(\lambda \right)$ analytically depend in $\lambda $ in any points of $\mathcal{B}^1=\mathbb{R}^{1} \bigcap \mathcal{B}$ and be unitary in this points. 

Let $R\left(\lambda \right)$ (\ref{GEQ__63_}) correspond to characteristic operator $M\left(\lambda \right)$ (\ref{13}), (\ref{GEQ__64ad1_}) of equation (\ref{GEQ__5_}), where characteristic projection (\ref{GEQ__64ad1_})   corresponds to some Weyl function $m(\lambda)$ of equation (\ref{GEQ__51_}) and to pair (\ref{alambda}), (\ref{barlambda}) which is constructed with the help of this $v\left(\lambda \right)$. Let the generalized spectral family $E_\mu$ correspond to $R\left(\lambda \right)$ by (\ref{GEQ__3_}).

Let $m_{a,b} \left(\lambda \right)$ be Nevanlinna operator-function from Lemma \ref{lm12} corresponding by (\ref{GEQ__64ad1_}), (\ref{GEQ__68ad1_}), (\ref{alambda}), (\ref{barlambda}) to this $v\left(\lambda \right)$. Let $\sigma _{a,b} \left(\mu \right)=w-\mathop{\lim }\limits_{\varepsilon \downarrow 0} \frac{1}{\pi } \int_{0}^{\mu }{\Im}m_{a,b} \left(\mu +i\varepsilon \right) d\mu $ be the spectral operator-function that corresponds to $m_{a,b} \left(\lambda \right)$. 

Then every proposition of Theorem \ref{th7} is valid with $\sigma _{a,b} \left(\mu \right)$ instead of $\sigma \left(\mu \right)$, $\left(u_{1} \left(t,\lambda \right),\ldots ,u_{n} \left(t,\lambda \right)\right)$ instead of $\left[X_{\lambda } \left(t\right)\right]_{1} $ and
\begin{multline*}\varphi \left(\mu ,f\right)=\int _{\mathcal{I}}\left(u_{1} \left(t,\mu \right),\ldots ,u_{n} \left(t,\mu \right)\right)^{*} m\left[f\left(t\right)\right] dt=\\=\int _{\mathcal{I}}\sum _{k=0}^{{s\mathord{\left/ {\vphantom {s 2}} \right. \kern-\nulldelimiterspace} 2} }\left(u_{1}^{\left(k\right)} \left(t,\mu \right),\ldots ,u_{n}^{\left(k\right)} \left(t,\mu \right)\right)^{*}  \mathrm{m}_{k} \left[f\left(t\right)\right]dt
\end{multline*}
instead of $\varphi(\mu,f)$ \eqref{GEQ__85_}, where $u_{j} \left(t,\lambda \right)$ see (\eqref{GEQ__115_}), $\mathrm{m}_{k} \left[f\left(t\right)\right]$ see \eqref{mk}.

Therefore if we represent spectral operator-function $\sigma _{ab} \left(\mu \right)$ in matrix form: $\sigma _{a,b} \left(\mu \right)=\left\| \left(\sigma _{ab} \left(\mu \right)\right)_{ij} \right\| _{i,j=1}^{n} ,\, \left(\sigma _{a,b} \left(\mu \right)\right)_{ij} \in B\left(\mathcal{H}\right)$ then, for example\footnote{Also \eqref{GEQ__84_}, Parseval equality \eqref{GEQ__88_}, Bessel inequality \eqref{GEQ__89_} can be represented in a similar way}, the following inversion formula is valid in $L_{m}^{2} \left(0,\infty \right)$ for any vector-function $f(t)\in \overset{\, \circ}H$ satisfying \eqref{star3} :
\begin{gather*}
f\left(t\right)=
\mathop{P}\limits^{\circ } \int _{\mathcal{B}^1}\sum _{i,j=1}^{n}u_{i} \left(t,\mu \right) d\left(\sigma _{a,b} \left(\mu \right)\right)_{ij}   \int _{\mathcal{I}}u_{j}^{*} \left(s,\mu \right)m\left[f\left(s\right)\right] ds=\\=\mathop{P}\limits^{\circ } \int _{\mathcal{B}^1}\sum _{i,j=1}^{n}u_{i} \left(t,\mu \right) d\left(\sigma _{a,b} \left(\mu \right)\right)_{ij}   \int _{\mathcal{I}}\sum\limits_{k=0}^{s/2}\left(u_{j}^{\left(k\right)} \left(s,\mu \right)\right)^*\mathrm{m}_{k} \left[f\left(s\right)\right] ds.
\end{gather*}
\end{theorem}
 
The proof is carried out in the same way as the proofs of Theorem \ref{th7} and Proposition \ref{prop21} with the help of Proposition \ref{rm21} and Lemma \ref{lm12}. 

Let us notice that in contrast to operator spectral function $\sigma_{a,b}(\mu)$ from Theorem \ref{th24} the scalar spectral function in \cite{Shtraus2} was constructed with the help of different formulae that corresponds to such intervals of real axis where $v\left(\mu \right)\ne -1$ or $v\left(\mu \right)\ne 1$. But already in matrix case it is impossible to construct the spectral matrix according to \cite{Shtraus2} since here for some real $\lambda$ (and even for any real $\lambda$) $(v(\lambda)+I_n)^{-1}$ and $(v(\lambda)-I_n)^{-1}$ may simultaneously do not belong to $B(\mathcal{H}^n)$.

\begin{remark}
Let contraction $v\left(\lambda \right)\in B\left(\mathcal{H}^{n} \right)$ is analytical in any point $\lambda \in \bar{{\mathbb C}}_{+} \bigcap \mathcal{B}$ and is unitary in any point $\lambda \in \mathcal{B}^1\not=\varnothing $. Let $\dim \mathcal{H}<\infty $. Then $v\left(\lambda \right)$ satisfy conditions of Lemma \ref{lm12}.

 If $\dim \mathcal{H}=\infty $ in general it is not valid. Namely let domain $D\subset {\mathbb C}_{+} $, $dist\left\{\bar{D},\, \left[-a,a\right]\right\}>0\, \, \forall a\in {\mathbb R}_{+}^{1} $; $set\left\{\lambda _{k} \right\}_{k=1}^{\infty } \subset D$ in dense in $D$. Let us consider in $\mathcal{H}^{n} =\left(l^{2} \right)^{n} $ the following operator $v\left(\lambda \right)=v_1\left(\lambda \right)\oplus I_{n-1} $, where $v_1\left(\lambda \right)=diag\left\{\frac{\lambda -\lambda _{k} }{\lambda -\bar{\lambda }_{k} } \right\}_{k=1}^{\infty } $. This operator is analytical in any point $\lambda \in \bar{{\mathbb C}}_{+} $, is a contraction for $\lambda \in {{\mathbb C}}_{+} $ and is unitary for $\lambda \in {\mathbb R}^{1} $. But for this operator the set $S=\left\{\lambda \in {\mathbb C}_{+} :v^{-1} \left(\lambda \right)\notin B\left(\mathcal{H}^n\right)\right\}=\bar{D}$.

\end{remark}

Finally, we note that, obviously, an analogue of Theorem \ref{th24} is valid for $\mathcal{I}=(0,b)$, $b<\infty$.


\begin{thebibliography}{99}

\bibitem{0} {S. Albeverio, M. Malamud, V. Mogilevskii}, \textit{On Titchmarsh-Weyl functions of first-order symmetric systems with arbitrary deficiency indices},	{arXiv:1206.0479}.

\bibitem{Atkinson}  F.~Atkinson, \textit{Discrete and Continuous Boundary
Problems}, Acad. Press, New York-London, 1964. (Russian
translation: Mir, Moscow, 1988).

\bibitem{Berez1}  Ju. M.~Berezanskii, \textit{Expansions in Eigenfunctions of Selfadjoint Operators, }Amer. Math. Soc., Providence, 1968. (Russian edition: Naukova Dumka, Kiev, 1965).

\bibitem{Berez2}  Yu.M.~Berezanskii, \textit{Selfadjoint Operators in Spaces of Infinitely Many Variables, }Amer. Math. Soc., Providence, 1986. (Russian edition: Naukova Dumka, Kiev, 1978).

\bibitem{Bruk1}  V.M.~Bruk, \textit{The generalized resolvents and spectral functions of even order differential operators in a space of vector-valued functions}, Mat. Zametki \textbf{15} (1974), no.~6, 945--954. (Russian); English transl.: Math. Notes \textbf{15} (1974), 563--568.


\bibitem{CodLev} E. Coddington, N. Levinson, \textit{Theory of ordinary differential equations}, McGraw-Hill Book Company, New York--Toronto--London, 1955.


\bibitem{DerMalamud}  V.A.~Derkach, M.M.~Malamud, \textit{Generalized resolvents and the boundary value problems for Hermitian operators with gaps,} J. Funct. Anal. \textbf{95} (1991), no.~1, 1--95.

\bibitem{DHMdS1} V. Derkach, S. Hassi, M. Malamud, H. De Snoo,
\textit{Boundary relations and their Weyl families}, Trans. Am.
Math. Soc. \textbf{358} (2006), no.~12, 5351--5400.

\bibitem{DHMdS2} V. Derkach, S. Hassi, M. Malamud, H. De Snoo, \textit{Boundary
relations and generalized resolvents of symmetric operators},
Russ. J. Math. Phys. \textbf{16} (2009), no.~1, 17--60.

\bibitem{DSnoo1}  A. Dijksma, H. de Snoo, \textit{Self-adjoint extensions of symmetric subspaces}Pacific J. Math. \textbf{54} (1974), 71--100.


\bibitem{DS}  N. Dunford, D. Schwartz, \textit{Linear operators. Part II: Spectral theory. Self adjoint operators in Hilbert space, }Mir, Moscow, 1966. (Russian),

\bibitem{Gor} M.L. ~Gorbachuk,  \textit{On spectral functions of a second order differential equation with operator coefficients}, Ukrain. Mat. Zh. \textbf{18} (1966), no.~2, 3--21. 

\bibitem{GorGor}  V.I.~Gorbachuk, M.L.~Gorbachuk, \textit{Boundary Value Problems for Operator Differential Equations, }Kluwer Academic Publishers, Dordrecht--Boston--London, 1991. (Russian eqition: Naukova Dumka, Kiew, 1984).

\bibitem{Halmos} P.R.~Halmos, \textit{A Hilbert space problem book}, Moscow, Mir, 1970.
      

\bibitem{Kato}  T.~Kato, \textit{Perturbation Theory for Linear Operators, }Springer--Verlag, Berlin--Heidelberg--New York, 1966. (Russian translation: Mir, Moscow, 1972).

\bibitem{Khrab1}  V.I.~Khabustovskiy, \textit{The spectral matrix of a
periodic symmetric system with a degenerate weight on the axis,
}Teor. Funktsii Funktsional. Anal. i Prilozhen. \textbf{35} (1981),
111--119. (Russian).

\bibitem{Khrab2}  V.I.~Khrabustovskiy, \textit{Eigenfunction expansions of periodic systems with weight, }Dokl. Akad. Nauk Ukrain. SSR Ser. A (1984), no.~5, 26--29. (Russian).

\bibitem{Khrab3}  V.I.~Khrabustovskiy, \textit{Spectral analysis of periodic systems with degenerate weight on the axis and half-axis, }Teor. Funktsional. Anal. i Prilozhen. \textbf{44} (1985), no.~4, 122--133. (Russian); English transl.: J. Soviet Math. \textbf{48} (1990), no.~3, 345--355.

\bibitem{Khrab4}  V.I.~Khrabustovskiy, \textit{On the characteristic matrix of Weyl-Titchmarsh type for differential-operator equations with the spectral parameter which contains the spectral parameter in linear or Nevanlinna's manner, }Mat. Fiz. Anal. Geom. \textbf{10} (2003), no.~2, 205--227. (Russian).

\bibitem{Khrab5}  V.I.~Khrabustovskiy,
\textit{On the Characteristic Operators and Projections and on the Solutions of Weyl
Type of Dissipative and Accumulative Operator Systems. }I.\textit{ General Case, }
J. Math. Phys. Anal. Geom. \textbf{2} (2006), no.~2, 149--175; II. \textit{Abstract theory, }ibid, no.~3, 299--317; III. \textit{Separated boundary conditions, }ibid, no.~4, 449--473.

\bibitem{Khrab+}  V.I.~Khrabustovskyi, \textit{On the limit of
regular dissipative and self-adjoint boundary value problems with
nonseparated boundary conditions when an iterval streches to the
semiaxis}, J. Math. Phys. Anal. Geom. \textbf{5} (2009), no.~1, 54--81.

\bibitem{Khrab6} V.I.~Khrabustovskiy,
\textit{Expansion in eigenfunctions of relations generated by pair
of operator differential expressions,} Methods Funct. Anal. Topol.
\textbf{15} (2009), no.~2, 137--151.

\bibitem{KhrabArxiv} V.I.~Khrabustovskiy, \textit{Analogs of generalized resolvents and eigenfunction
expansions of relations generated by pair of differential operator
expressions one of which depends on spectral parameter in
nonlinear manner},  arXiv:1301.2926.

\bibitem{KhrabMAG} V.I.~Khrabustovskiy, \textit{Analogs of generalized resolvents of relations
generated by pair of differential operator expressions one of which depends on spectral parameter in nonlinear manner}, J. Math. Phys. Anal. Geom. \textbf{9} (2012), no.~4. 


\bibitem{LyaSto}  V.E.~Lyantse, O.G.~Storozh, \textit{Methods of the Theory of Unbounded Operators, }Naukova Dumka, Kiev, 1983. (Russian).

\bibitem{Marchenko}  V.A.~Marchenko, \textit{Sturm-Liouville Operators and Their Applications, }Oper. Theory Adv. Appl., vol. 22, Birkhauser, Basel, 1986. (Russian edition: Naukova Dumka, Kiev, 1977).

\bibitem{Mogil} V. Mogilevskii, \textit{Boundary pairs and boundary conditions for general (not necessarily definite) first-order symmetric systems with arbitrary deficiency indices}, {Math. Nachr.} \textbf{285} (2012), no.~14-15, 1895-1931.

\bibitem{Naimark} M.A. Naimark, \textit{Linear Differential Operators}, Nauka, Moscow, 1969.

\bibitem{Orlov} S.A. Orlov, \textit{Description of the Green functions of canonical differential
systems. I.II.} {J. Sov. Math.} \textbf{52} (1990), no.~6, 3500--3508; no.~5,
3372--3377; translation from Teor. Funkts., Funkts. Anal. Prilozh.
\textbf{51} (1989), 78--88; \textbf{52}(1989), 33--39.

\bibitem{RB60} F.S. Rofe-Beketov, \textit{Expansion in eigenfunctions of infinite systems of differential equations in the non-self-adjoint and self-adjoint cases}, Mat. Sb. \textbf{51} (1960), no.~3, 293--342. 

\bibitem{RB} F.S. Rofe-Beketov, \textit{Self-adjoin extensions of
differential operators in a space of vector-valued functions},
{Teor. Funkts., Funkts. Anal. Prilozh.} \textbf{8} (1969), 3--24.

\bibitem{RBUpsala} F.S. Rofe-Beketov, \textit{Square-integrable solutions, self-adjoint
extensions and spectrum of differential systems}, Differ. Equat.,
Proc. int. Conf., Uppsala 1977, 169--178.

\bibitem{RBKholkin}  F.S.~Rofe-Beketov and A.M.~Khol'kin,
\textit{Spectral Analysis of Differential Operators. Interplay Beetween
Spectral and Oscillatory Properties, }World Scientific Monograph Series in Mathematics,
Vol. 7, NJ, 2005.

\bibitem{Sakhno} L.A.~Sakhnovich, \textit{Spectral Theory of Canonical Differential Systems. Method of Operator Idetities, }Oper. Theory Adv. App., vol. 107, Birkhauser, Basel, 1999.


\bibitem{Shtraus2}   A.V. Shtraus, \textit{ On eigenfunction expansion of a second order boundary problem of a semi-axis}, Izv. Akad. Nauk SSSR. Ser. Mat. \textbf{20} (1956), 783--792.

\bibitem{Shtraus1}  A.V.~Shtraus, \textit{On the generalized resolvents and spectral functions of even-order differential operators, }Izv. Akad. Nauk SSSR Ser. Mat. \textbf{21} (1957), no.~6, 785--808. (Russian); English transl.: Amer. Math. Soc. Transl. Ser. 2 \textbf{16} (1960), 462--464.


\bibitem{Nagy}B. Sz.-Nagy, C. Foias, Analyse harmonique des ope’rateurs de l’espace de Hilbert, Masson et Cie, Paris, 1967. (Russian translation: Mir, Moskow, 1970).


\end{thebibliography}
\end{document}